\documentclass[11pt]{article}
\usepackage{amsfonts}
\usepackage{amssymb}
\usepackage{amsthm}
\usepackage{amsmath}
\usepackage{graphicx}
\usepackage{empheq}
\usepackage{indentfirst}
\usepackage{cite}
\usepackage{mathrsfs}
\usepackage{graphics}
\usepackage{enumerate}

 \newtheorem{theorem}{Theorem}[section]
 \newtheorem{corollary}{Corollary}[section]
 \newtheorem{lemma}{Lemma}[section]
 \newtheorem{proposition}{Proposition}[section]
 \newtheorem{definition}{Definition}[section]
 \newtheorem{remark}{Remark}[section]
 \numberwithin{equation}{section}

\allowdisplaybreaks[4]

\def\ddiv{\mathrm{div}}

\def\e{\epsilon}

\newcommand{\beq}{\begin{equation}}
\newcommand{\eeq}{\end{equation}}
\def\non{\nonumber }
\def\bea{\begin{eqnarray}}
\def\eea{\end{eqnarray}}

\def\ub{\mathbf{u}}

\begin{document}

\title{Well-posedness and long-time behavior of a non-autonomous Cahn-Hilliard-Darcy system with mass source modeling tumor growth}
\author{Jie Jiang \thanks{Wuhan
Institute of Physics and Mathematics, Chinese Academy of Sciences,
Wuhan 430071, HuBei Province, P.R. China,
\textsl{jiangbryan@gmail.com}}, \
Hao Wu\thanks{School of Mathematical Sciences and Shanghai Key Laboratory for Contemporary Applied
Mathematics, Fudan University, Shanghai 200433,\ P.R. China,
\textsl{haowufd@yahoo.com}}\quad and\ \
Songmu Zheng \thanks{Institute of Mathematics, Fudan University, Shanghai 200433,\ P.R. China,
\textsl{songmuzheng@yahoo.com}}
}

\date{\today}
\maketitle
\begin{abstract}
In this paper, we study an initial boundary value problem of the Cahn-Hilliard-Darcy system with a non-autonomous mass source term $S$ that models tumor growth.
We first prove the existence of global weak solutions as well as the existence of unique local strong solutions in both 2D and 3D.
Then we investigate the qualitative behavior of solutions in details when the spatial dimension is two. More precisely, we prove that the strong solution exists globally and it defines a closed dynamical process. Then we establish the existence of a minimal pullback attractor for translated bounded mass source $S$. Finally, when $S$ is assumed to be asymptotically autonomous, we demonstrate that any global weak/strong solution converges to a single steady state as $t\to+\infty$. An estimate on the convergence rate is also given.\medskip
\\
\noindent  {\bf Keywords}: Cahn-Hilliard-Darcy system; non-autonomous; well-posedness; long-time behavior.\\
\noindent \textbf{AMS Subject Classification}: 35G25, 35J20, 35B40, 35B45
\end{abstract}

%%%%%%%%%%%%%%%%%%%%%%%%%%%%%%%%%%%%%%%%%%%%%%%%%%%%%%%%%%%%%%%%%
\section{Introduction}

In this paper, we consider the following Cahn-Hilliard-Darcy (CHD in short) system that arises in the study of morphological evolution in solid tumour growth (see, e.g., \cite{FJCWLC,Wise2008}):
 \begin{eqnarray}
  && \phi_t+\ddiv(\ub\phi)=\Delta\mu+S,\quad \text{in}\ \ (\tau,T)\times \Omega,\label{1}\\
  && \mu= -\e^2\Delta \phi+f'(\phi) \quad \text{with} \quad f(\phi)=\frac{1}{4}\phi^4-\frac{1}{2}\phi^2,\label{mu}\\
  && \ub=-\nabla p+\frac{\gamma}{\e}\mu\nabla\phi,\qquad \quad \text{in} \ \ (\tau,T)\times \Omega,\label{2}\\
  && \ddiv \ub=S,\qquad \qquad \qquad \quad \text{in}\ \ (\tau,T)\times \Omega.\label{3}
\end{eqnarray}
Here, $\Omega$ is assumed to be a bounded domain in $\mathbb{R}^d$ ($d\in\{2,3\}$). $\tau\in \mathbb{R}$ denotes the initial time and $T>\tau$ is any given number. The CHD system \eqref{1}--\eqref{3} is subject to the following boundary and initial conditions:
 \bea
 && \partial_{\nu}\phi=\partial_{\nu}\mu=0,\quad\text{on }
\partial\Omega, \label{b1}\\
 && \ub\cdot\nu=0, \qquad\text{on }\partial\Omega, \label{b2}\\
 && \phi(t,x)|_{t=\tau}=\phi_\tau(x),\label{ini}
 \eea
 where $\nu$ is the unit outward normal vector to the boundary $\partial\Omega$.

The CHD system \eqref{1}--\eqref{3} can be viewed as the simplest version of those general diffuse interface models for tumor growth, which were derived based on the principle of mass conservation together with the second law of thermodynamics \cite{FJCWLC,Wise2008}. In the diffuse-interface (or phase-field) framework, the tumor volume fraction is denoted by a scalar order parameter $\phi$ and the sharp tumor/host interfaces are replaced by narrow transition layers, whose thickness is approximately characterized by a small parameter $\e>0$. Instead of tracking the interfaces explicitly, the dynamics of interfaces (now recognized as zero level sets of the order parameter) can be simulated on a fixed grid. Therefore, the diffuse-interface model has the advantage that it can easily describe topological transitions of interfaces (e.g., pinch-off and reconnection for two phase immiscible flow) in a natural way (see \cite{AMW, Good1, Good2, GPV, HH}).

Equation \eqref{1} is a convective Cahn-Hilliard type equation, which is derived from the mass conservation. The vector $\ub$ stands for the advective velocity field, while  the scalar functions $\mu$, $S$ stand for the chemical potential and the mass source term accounting for cell proliferation (or the rate of change in tumor volume, see \cite{FJCWLC,Wise2008}), respectively. The chemical potential $\mu$ is the variational derivative of the free energy functional:
 \beq
 E(\phi):=\int_\Omega\left( \frac{\e^2}{2}|\nabla\phi|^2+f(\phi)\right)dx,\non
  \eeq
  in which the function $f$ (see \eqref{mu}) can be viewed as a smooth double-well polynomial approximation of the physically relevant logarithmic potential (see \cite{58}).  Equation \eqref{2} for the advective velocity $\ub$ follows from a generalized Darcy's law, in which $\gamma$ is a positive constant measuring the excess adhesion force at the diffusive tumor/host tissue interfaces and $p$ is the pressure that consists of a combination of certain generalized Gibbs free energy and the gravitational potential. Equation \eqref{3} serves as a constraint for the velocity due to the possible mass exchange.

 We recall some previous works in the literature that are related to our problem. In biological applications, e.g., the phase-field models for tumour growth and wound healing \cite{KS08,FJCWLC}, the mass source term $S$ may depend on the order parameter $\phi$ in a quadratic way such that $S=\alpha \phi(1-\phi)$ ($\alpha>0$). When $S$ has a linear dependence on $\phi$, Equation \eqref{1} (neglecting the velocity $\ub$) is also known as the Cahn-Hilliard-Oono equation that accounts for long-range (nonlocal) interactions
in the phase separation process \cite{OP87}. Concerning the mathematical analysis for these generalized Cahn-Hilliard equations with mass source (with the convection under velocity $\ub$ being neglected), we refer to the recent work \cite{Mi11, Mi13, Mi14}, in which well-posedness and asymptotic behavior of the associated dynamical system have been investigated. When $S=0$, the CHD system \eqref{1}--\eqref{3} is referred to as the Cahn-Hilliard-Hele-Shaw (CHHS) system that has been used to describe two-phase flows in the Hele-Shaw geometry \cite{Good1,Good2} (see also \cite{SO92} for a similar model for spinodal decomposition of a binary fluid in a Hele-Shaw cell). The CHHS system with zero mass source term has been studied by many authors in the literature, both numerically and mathematically. For instance, an unconditionally energy stable and solvable finite difference scheme based on convex-splitting was proposed in \cite{Wise2010}, see also \cite{FW12} for an implicit Euler temporal scheme combined with a mixed finite element discretization in space. Concerning the analysis results, existence and uniqueness of global classical solutions in 2D torus and local classical solution in 3D torus were first established in \cite{WZ2013}. Besides, some blow-up criteria were also obtained in the three dimensional case. In \cite{WW2012}, long-time behavior of global solutions and stability of local minimizers in both 2D and 3D periodic setting were proved based on the \L ojasiewicz-Simon approach \cite{S83}. For the CHHS system in a 2D rectangle or in a 3D box under homogeneous Neumann boundary conditions, qualitative behaviors of strong solutions such as existence, uniqueness, regularity and asymptotic stability of the constant state $\frac{1}{|\Omega|}\int_\Omega \phi_\tau dx$ are studied in \cite{LTZ}. Quite recently, the connection between the Cahn-Hilliard-Brinkman (CHB) system and the CHHS system has been investigated in \cite{BCG} such that a suitable weak solution to the CHHS system can be shown to be a limit of solutions to the CHB system as the fluid viscosity goes to zero. Moreover, we would like to remark that the CHHS system can be viewed as a simplification of the full Cahn-Hilliard-Navier-Stokes (CHNS) system (see e.g., \cite{AMW,GPV,HH}) in the Hele-Shaw geometry. We refer to \cite{Ab,AAB,GG10,CFG,GG10a,Me12,ZWH} and the references therein for analytical results of the CHNS system on well-posedness as well as long-time behavior under various situations.

 However, to the best of our knowledge, there seems no analytical results in the literature concerning the CHD system \eqref{1}--\eqref{3} with a non-zero mass source term $S$. This is the main goal of the present paper. In this paper, we shall confine ourselves to the situation that $S$ is assumed to be a given source of mass, possibly depending on time $t$ and position $x$, but not on the parameter $\phi$. The case with more general mass source term will be treated in the future work.

 We summarize the main results of this paper as follows. First, under suitable integrability conditions on the mass source term $S$, we apply the Galerkin method to prove the existence of global weak solutions as well as the existence and uniqueness of local strong solutions to the CHD system \eqref{1}--\eqref{ini} in both 2D and 3D cases (see Theorem \ref{Th1}). Then we focus on the studies of qualitative behavior for solutions in the 2D case. It is shown that in 2D, problem \eqref{1}--\eqref{ini} actually admits a unique global strong solution $\phi$ in $H^2_N(\Omega)$ which defines a family of \emph{closed processes} $\{U(t,\tau)\}_{t\geq \tau}$ on $H^2_N(\Omega)$ (see Theorem \ref{Th2}). If the mass source $S$ is further assumed to be a \emph{translated bounded} function in $L^2_tL^2_x$ (see \eqref{detrans}), the family of processes $\{U(t,\tau)\}_{t\geq \tau}$ that are confined on the phase space $\mathcal{H}_{M}$ (see \eqref{H0}) turns out to admit a minimal \emph{pullback attractor} $\mathcal{A}$ (see Definition \ref{pull} and  Theorem \ref{T2}). In addition, we prove that under suitable decay assumption on $S$ (see \eqref{assS}), the dynamical process becomes \emph{asymptotically autonomous}. In this specific case, the $\omega$-limit set of each trajectory is actually a singleton. Namely, for arbitrary large initial datum, the global bounded solution will converge to a single steady state as $t\to+\infty$ and an estimate on the convergence rate is also given (see Theorem \ref{T3}).

Before concluding the introduction part, we would like to stress some new features of the present paper.
The presence of the mass source term $S$ brings us several difficulties in the mathematical analysis. First, unlike in \cite{FW12,WZ2013,WW2012,LTZ}, the velocity field $\ub$ is no longer divergence free. As a consequence, in order to prove the existence of weak/strong solutions, we use a modified Galerkin approximation different from that in \cite{LTZ}. Instead of solving the approximate velocity directly (by taking the Helmholtz-Leray orthogonal projection to eliminate the pressure term), we solve the pressure function that satisfies a Poisson type equation subject to homogeneous Neumann boundary condition (see \eqref{p}) and then obtain the velocity via the Darcy equation \eqref{2}. Besides, some new estimates for the pressure $p$ and its derivative (cf. \cite{WZ2013}) are derived, which play an important role in the subsequent proofs for existence of global solutions (see Lemma \ref{espp}).

Second, we study the long-time dynamics of problem \eqref{1}--\eqref{ini}
from the infinite dimensional dynamical system point of view \cite{TE}. The theory of global attractors has been generalized to the case of non-autonomous dynamical systems, for instance, the uniform attractors (see \cite{CV}) and pullback attractors (see \cite{CDF97, KR11} and the references therein). In this paper, we prove the existence of a pullback attractor for the CHD system \eqref{1}--\eqref{ini} under rather general assumptions on the time dependent mass source term $S$ in 2D.
Due to the mass conservation property \eqref{mm}, we cannot expect an absorbing set for initial data varying in the whole space. Instead, we first confine the associated dynamical process $\{U(t,\tau)\}_{t\geq \tau}$  on a suitable phase space $\mathcal{H}_M$ (see \eqref{H0}), which is a subset of $H^2_N(\Omega)$. Next, due to the highly nonlinear coupling of the CHD system, it seems difficult to obtain (strong) continuity of the process $\{U(t,\tau)\}_{t\geq \tau}$ in $\mathcal{H}_M$ but only a continuous dependence result in the lower-order space $H^1$ (see Lemma \ref{pconti}). This indicates that the process $\{U(t,\tau)\}_{t\geq \tau}$ is only \emph{closed} (see Definition \ref{clodef}, cf.  also \cite{PZ07} for the notion of closed semigroups). We then perform a nonstandard argument devised in \cite{GMR12} for closed processes to conclude our result (cf. \cite{Me10} for the case with closed cocycles). For this purpose, we deduce a generalized Gronwall type inequality (see Lemma \ref{Gron1}) to obtain some uniform estimates that lead to the existence of a pullback absorbing set (see Proposition \ref{plab}). We believe that Lemma \ref{Gron1} may have its own interests and can be applied to other problems with highly nonlinear structure. Besides, since the mass source term $S$ is only assumed to be translated bounded in $L^2_tL^2_x$, we are not able to obtain higher-order estimates of the solutions (and thus compactness) by taking derivatives of the PDEs. Instead, we use a continuity method for energy functions (see e.g., \cite{GMR12,MPR09}) to obtain the pullback asymptotic compactness (see Proposition \ref{compact}).

       At last, we study the long-time behavior for any bounded global weak/strong solution of the CHD system \eqref{1}--\eqref{ini} when the mass source $S$ becomes asymptotically autonomous. This is nontrivial, since the topology of the set of steady states (see \eqref{SSS}) can be rather complicated in high dimensional case and it may form a continuum (see e.g., \cite{CPDE1999}). Moreover, since our problem \eqref{1}--\eqref{ini} is now non-autonomous due to the presence of $S$, it no longer has a Lyapunov functional. Nevertheless, for global bounded solutions in $H^2$, it is possible to derive an energy inequality (see \eqref{BEL}), which enables us to characterize the corresponding $\omega$-limit sets. Based on that energy inequality, we are able to apply the \L ojasiewicz-Simon approach (cf. \cite{S83, CJ, FS, HT01}) to obtain the convergence of $\phi(t)$ as time goes to infinity as well as an estimate on convergence rate. Our convergence result generalizes the previous one in \cite{WW2012} for the homogeneous CHHS system in periodic setting. Moreover, we do not need to impose any additional assumption either on the initial datum for $\phi$ (e.g., the average of initial datum $\frac{1}{|\Omega|}\int_\Omega \phi_\tau dx$ being outside the spinodal region) or on the size of domain (being 'small') like in \cite{LTZ} in order to obtain certain asymptotical stability.

The rest of this paper is organized as follows. In Section 2, we introduce the functional settings and state the main results of this paper. Section 3 is devoted to the proof of the existence of global weak solutions as well as the existence and uniqueness of local strong solutions to problem \eqref{1}--\eqref{ini} in both 2D and 3D. In Section 4, we prove the existence of a unique global strong solution as well as the regularity of weak solutions in 2D. Then we show in Section 5 that the associated closed processes $\{U(t,\tau)\}_{t\geq \tau}$ on the phase space $\mathcal{H}_{M}$ admit a minimal pullback attractor $\mathcal{A}$, provided that the mass source $S$ is translated bounded in $L^2_tL^2_x$. Finally, in Section 6, we prove the convergence of global weak/strong solutions to a single steady state as $t\to +\infty$ and obtain an estimate on the convergence rate.

\section{Preliminaries and Main Results}
We first introduce some notations on the functional spaces. Let $\Omega\subset \mathbb{R}^d$, $d=2,3$, be either a smooth bounded domain or a convex polygonal or polyhedral domain. $L^q(\Omega)$, $1 \leq q \leq\infty$ denotes the usual
Lebesgue space and $\|\cdot\|_{L^q(\Omega)}$ denotes its norm. Similarly, $W^{m,q}(\Omega)$, $m \in \mathbb{N}$, $1 \leq q \leq \infty$, denotes the usual Sobolev space with norm $\|\cdot \|_{W^{m,p}(\Omega)}$. When $q=2$, we simply denote $W^{m,2}(\Omega)$ by $H^m(\Omega)$ and denote the norms $\|\cdot\|_{L^2(\Omega)}$, $\|\cdot\|_{H^m(\Omega)}$ by $\|\cdot\|$ and $\|\cdot\|_{H^m}$, respectively. The $L^2$-Bessel potential spaces are denoted by $H^s(\Omega)$, $s\in \mathbb{R}$, which are defined by restriction of distributions in $H^s(\mathbb{R}^d)$ to $\Omega$. If $X$ is a Banach space, we denote by $X'$ its dual and by $\langle \cdot, \cdot\rangle$ the associated duality product. The inner product in $L^2$ will be denoted by $(\cdot, \cdot)$. If $I$ is an interval of $\mathbb{R}^+$ and $X$ a Banach space, we use the function space $L^p(I;X)$, $1 \leq p \leq +\infty$, which consists of $p$-integrable
functions with values in $X$. Moreover, $C_w(I;X)$ denotes the topological vector space of all bounded and weakly continuous functions from $I$ to $X$, while $W^{1,p}(I,X)$ $(1\leq p<+\infty)$ stands for the space of all functions $u$ such that $u, \frac{du}{dt}\in L^p(I;X)$, where $\frac{du}{dt}$ denotes the vector valued distributional derivative of $u$. Bold characters will be used to denote vector spaces.

Given any function $v \in L^1(\Omega)$, we denote by $\overline{v} = |\Omega|^{-1}\int_\Omega v(x)dx$ its mean value. Then we define the space
$\dot L^2(\Omega) := \{ v \in L^2(\Omega): \overline{v}=0\}$ and $\dot v=\mathrm{P}_0 v := v - \overline{v}$ the orthogonal projection onto $\dot L^2(\Omega)$. Furthermore, we denote $\dot H^1(\Omega)=H^1(\Omega)\cap \dot L^2(\Omega)$, which is a Hilbert space with inner product $(u, v)_{\dot H^1}=\int_\Omega \nabla u\cdot \nabla v dx$ due to the classical Poincar\'e inequality for functions with zero mean. Its dual space is simply denoted by $\dot{H}^{-1}(\Omega)$.
Denote the spaces
$H^2_N=\{\varphi\in H^2(\Omega)\;|\;\;\partial_{\nu}\varphi=0 \;\;\text{on}\;\;\partial\Omega\}$ and $H^4_N=\{\varphi\in H^4(\Omega)\;|\;\;\partial_{\nu}\varphi=\partial_{\nu} \Delta \varphi=0 \;\;\text{on}\;\;\partial\Omega\}$. We can see that the operator $A=-\Delta$ with its domain $D(A)=H^2_{N}\cap \dot L^2(\Omega)$ is a positively defined, self-adjoint operator on $D(A)$ and the spectral theorem enables us to define powers $A^s$ of $A$ for $s\in\mathbb{R}$. Then space $(H^1(\Omega))'$ is endowed with the equivalent norm $\|v\|^2_{H^1(\Omega)'}=\|A^{-\frac12}(v-\overline{v})\|^2+|\overline{v}|^2$ and the norm on $\dot{H}^{-1}(\Omega)$ is given by $\|v\|^2_{\dot H^{-1}}=\|A^{-\frac12}(v-\overline{v})\|^2$.

Throughout the paper, without loss of generality, we assume that $\gamma=\epsilon=1$. $C\geq 0$ will stand for a generic constant and $\mathcal{Q}(\cdot)$ for a generic positive monotone increasing function. Special dependence will be pointed out in the text if necessary.

Following the constraint \eqref{3} and the boundary condition \eqref{b2}, we can easily see that a necessary condition for the external force $S$ is that
\beq\label{nec}\int_{\Omega} S(t,x)dx\equiv0.\eeq

Below we introduce the definitions of weak solution as well as strong solution to the CHD system \eqref{1}--\eqref{3}.
\begin{definition} Assume $d=2,3$.

(i) Let $T>\tau$, $\phi_\tau\in H^1(\Omega)$ and $S\in L^{2}(\tau,T;
\dot{L}^2(\Omega))$ be given. A triplet $(\phi, \ub, p)$ is a weak solution to the system \eqref{1}--\eqref{3} endowed with boundary and initial conditions \eqref{b1}--\eqref{ini}, if
$$\phi\in C_w([\tau,T]; H^1(\Omega))\cap L^2(\tau,T; H^{3}(\Omega)),\quad \partial_t\phi\in L^{\frac{8}{5}}(\tau,T;(H^1(\Omega))'),$$
$$\ub\in L^2(\tau,T; \mathbf{L}^2(\Omega)), \ \ p\in L^\frac85(\tau, T; H^1(\Omega))$$ such that
\beq \langle \phi_t, \psi\rangle+\langle\ddiv(\ub \phi), \psi\rangle+(\nabla \mu, \nabla \psi)=(S,\psi),\ \ \forall\, \psi\in H^1(\Omega), \ \ \text{a.e.}\ t\in [\tau,T],\non\eeq
\beq (\nabla p, \nabla\varphi)= (S, \varphi)+(\mu\nabla\phi,\nabla\varphi),\quad \forall\, \varphi\in H^1(\Omega), \ \ \text{a.e.}\ t\in [\tau,T],\non
\eeq
\beq (\ub, \mathbf{v})= ( -\nabla p +\mu\nabla \phi, \mathbf{v}),\quad \forall\, \mathbf{v}\in \mathbf{L}^2(\Omega), \ \ \text{a.e.}\ t\in [\tau,T],\non
\eeq
with $\mu\in L^2(\tau, T; H^1(\Omega))$ given by \eqref{mu}, and
\beq \partial_\nu \phi=0,\quad \text{a.e. on}\ \partial\Omega\times(\tau, T),\non\eeq
\beq \phi|_{t=\tau}=\phi_\tau,\quad\text{a.e. in}\ \Omega.\non\eeq

(2) Let $T>\tau$, $\phi_\tau\in H^2_N(\Omega)$ and $S\in L^{2}(\tau,T;
\dot{L}^2(\Omega))$ be given. A triplet $(\phi, \ub, p)$ is a strong solution to the system \eqref{1}--\eqref{3} endowed with boundary and initial conditions \eqref{b1}--\eqref{ini}, if
$$\phi\in C([\tau,T];H^2_N(\Omega))\cap L^2(\tau,T;H^4_N(\Omega)),\quad  \phi_t\in L^2(\tau, T; L^2(\Omega)),$$
$$ \ub\in   L^2(\tau,T; \mathbf{H}^1(\Omega)),\quad p\in L^2(\tau, T; H^2(\Omega)),$$
$$\mu\in C([\tau, T]; L^2(\Omega))\cap L^2(\tau, T; H^2(\Omega)),$$
such that
\beq\phi_t+\ddiv(\ub\phi)=\Delta\mu+S,\quad \text{in}\ \ L^2(\Omega)
\quad\text{ a.e.} \ \ t\in [\tau, T]\non\eeq  with $\mu$ given by \eqref{mu},
\beq-\Delta p=S-\ddiv{(\mu\nabla\phi)}, \quad \text{in}\ \ L^2(\Omega)
\quad\text{ a.e.} \ \ t\in [\tau, T],
\non
 \eeq
\eqref{2} holds in $\mathbf{H}^1(\Omega)$ for a.e. $t\in[\tau, T]$ and
\beq \partial_\nu \phi=\partial_\nu \mu=\partial_\nu p=0,\quad \text{a.e. on}\ \partial\Omega\times(\tau, T),\non\eeq
\beq \phi|_{t=\tau}=\phi_\tau,\quad\text{a.e. in}\ \Omega.\non\eeq
\end{definition}
\begin{remark}
It is easy to see that the
 mean of any weak/strong solution $\phi$ over $\Omega$ is conserved in time, i.e.,
 \beq\label{mm}\overline{\phi}(t):=\frac{1}{|\Omega|}\int_{\Omega}\phi(t,x)dx\equiv\frac{1}{|\Omega|}\int_{\Omega}\phi_\tau dx:=M.\eeq
\end{remark}

Now we are in a position to state our main results.
\begin{theorem} \label{Th1} Suppose that $d=2,3$.

(i) For any $\phi_\tau \in H^1(\Omega)$ and $S\in L^{2}(\tau,T;
\dot{L}^2(\Omega))$ with arbitrary $T\in (\tau, +\infty)$, problem \eqref{1}--\eqref{ini} admits at least one global weak
solution $(\phi, \ub, p)$ on $[\tau, T]$.

(ii) For any $\phi_\tau \in H^2_N(\Omega)$, $S\in L^{2}(\tau,T;
\dot{L}^2(\Omega)) \cap L^{\infty}(\tau,T; \dot{H}^{-1}(\Omega))$ with arbitrary $T\in (\tau, +\infty)$, there exist a time $T^*\in(\tau, T)$ such that problem \eqref{1}--\eqref{ini} admits a strong solution $(\phi, \ub, p)$ on $[\tau, T^*]$ that is unique up to an additive function of $t$ to $p$.
\end{theorem}

When the spatial dimension is two, more comprehensive information about problem \eqref{1}--\eqref{ini} can be achieved. First, we can prove the existence of a unique global strong solution, i.e.,

\begin{theorem} \label{Th2} Suppose that $d=2$. For any $\phi_\tau \in H^2_N(\Omega)$, $S\in L^{2}_{loc}(\mathbb{R}; \dot{L}^2(\Omega))$ and arbitrary $T\in(\tau, +\infty)$, problem \eqref{1}--\eqref{ini} admits a global strong solution $(\phi, \ub, p)$ on $[\tau, T]$ that is unique up to an additive function of $t$ to $p$. The global strong solution defines a family of closed processes $\{U(t,\tau)\}_{t\geq \tau}$ on $H^2_N(\Omega)$ such that
\beq
\nonumber U(t,\tau)\phi_\tau=\phi(t),\quad \forall\, t\in [\tau, T].
\eeq
\end{theorem}
Consider the following phase space:
\beq
\label{H0}
\mathcal{H}_{M}=\left\{\phi\in H^2_N(\Omega),\ \ |\overline{\phi}|\leq M\right\}, \quad M \geq 0.
\eeq
For the external source term $S$, we consider the Banach space $L^2_b(\mathbb{R}; \dot L^2(\Omega))$ defined by
\beq
L^2_b(\mathbb{R}; \dot L^2(\Omega))=\left\{S\in L^2_{loc}(\mathbb{R};\dot L^2(\Omega))\,:\,\|S\|_{L^2_b(\mathbb{R};\dot L^2(\Omega))}^2:=\sup\limits_{t\in\mathbb{R}}\int_{t}^{t+1}\|S(s)\|^2ds<\infty\right\},\label{detrans}
\eeq
which is the subspace of $L^2_{loc}(\mathbb{R};\dot L^2(\Omega))$ of translation bounded functions.

Then we can prove that

\begin{theorem}\label{T2} Let $d=2$. For any $S\in L^2_{b}(\mathbb{R},\dot L^2(\Omega))$,  the family of closed processes $\{U(t,\tau)\}_{t\geq \tau}$ associated with problem \eqref{1}--\eqref{ini} defined on the phase space $\mathcal{H}_{M}$ admits a minimal pullback attractor $\mathcal{A}$ in the sense of Definition \ref{pull}.
\end{theorem}

Furthermore, if the dynamical process becomes \emph{asymptotically autonomous} under suitable assumptions on the external source $S$, we can prove that the global weak (or strong) solution converges to a single steady state as $t\to+\infty$ and obtain an estimate on the convergence rate.

\begin{theorem}\label{T3} Let $d=2$. Assume that $S\in L^2(\tau,+\infty;\dot L^2(\Omega))$ and satisfies the following condition
\beq
\sup\limits_{t\geq\tau}(1+t)^{1+\rho}\int_t^{+\infty}\|S\|^2 ds<+\infty,\quad\text{for some } \rho>0.\label{assS}
\eeq
Let $(\phi,\ub, p)$ be a global weak (or strong) solution to problem \eqref{1}--\eqref{ini}. Then there exists a steady state  $\phi_\infty \in H^2_N(\Omega)$, which is a solution to the stationary Cahn--Hilliard equation
\beq\label{stat}
\begin{cases}-\Delta\phi_{\infty}+f'(\phi_{\infty})=\int_{\Omega}f'(\phi_{\infty})dx,\quad\text{in } \Omega,\\
\partial_{\nu}\phi_{\infty}=0,\quad \text{on }\partial\Omega,\\
\int_{\Omega}\phi_{\infty}dx=\int_{\Omega}\phi_\tau dx
\end{cases}
\eeq
such that as $t\to+\infty$
$$
\begin{cases} \phi(t)\to \phi_\infty \quad \text{strongly in } H^s(\Omega), \ s<2,\\
\phi(t)\rightharpoonup \phi_\infty \quad \text{weakly in } H^2(\Omega).
\end{cases}
$$
Moreover, the following convergence rate holds
\beq
\|\phi(t)-\phi_{\infty}\|_{H^{s}}\leq C(1+t)^{-\frac{2-s}{3}\min\{\frac{\theta}{1-2\theta},\frac{\rho}{2}\}},\quad\forall\, t\geq \tau+1, \quad s\in [-1,2).\label{rate}
\eeq Here $C$ is a constant depending on $\|\phi_\tau\|_{H^1}$, $\int_\tau^{+\infty} \|S\|^2 d\tau$ and $\Omega$, $\theta\in(0,\frac{1}{2})$ is a constant depending on $\phi_{\infty}$.
\end{theorem}

\section{Well-posedness}

In this section, we prove Theorem \ref{Th1}, namely, the existence of global weak solutions and (unique) local strong solutions to the system \eqref{1}--\eqref{ini} in both 2D and 3D. For the sake of simplicity, we shall present the proofs in the 3D case, which are still valid for the 2D case with minor modifications due to different Sobolev embedding theorems and interpolation inequalities.

\subsection{Pressure estimate}

The following lemma on the estimate for the pressure $p$ will be useful in the subsequent analysis:

\begin{lemma} \label{espp}
Suppose $d=2,3$. For any given function $\phi \in H^3(\Omega)\cap H^2_N(\Omega)$,
the pressure function $p$ satisfies the following Poisson equation subject to a homogeneous Neumann boundary condition:
\beq
 \begin{cases}-\Delta p=S-\ddiv{(\mu\nabla\phi)}, \quad\text{in }
\Omega,\\
\partial_{\nu}p=0, \quad\text{on }\partial\Omega,\\
\int_\Omega p dx=0.
\end{cases}
\label{p}
 \eeq
Moreover, then the following estimates hold:
\beq
\|\nabla p\|\leq C\|S\|+C\|\mu\|_{L^6}\|\nabla \phi\|_{L^3},
\label{p01}
\eeq
\beq\label{p00} \|p\|\leq
C\|S\|+C\|\nabla\mu\|\|\nabla\phi\|_{L^{\frac{3}{2}}}+\left|\overline{\mu(\phi)}\right|\|\phi-\overline{\phi}\|,\eeq
where $\mu$ is given by $\mu=-\Delta\phi+\phi^3-\phi$.
\end{lemma}
\begin{proof}
It follows from the assumption on $\phi$ and the Sobolev embdedding theorem ($d=3$) that $\mu=-\Delta\phi+\phi^3-\phi\in H^1(\Omega)$.
Multiplying \eqref{p} by $p$ and integrating by parts, we get
$$\|\nabla p\|^2= \int_\Omega \left(S p +(\mu\nabla\phi)\cdot \nabla p\right) dx.$$
The above formula together with the Poincar\'e inequalty and the H\"older inequality easily yields \eqref{p01}.

Next, we deduce from \eqref{p} that
 \bea
p&=&A^{-1}S-A^{-1}\ddiv(\mu(\phi)\nabla\phi)\non\\
&=&A^{-1}S-A^{-1}\ddiv\left((\mu(\phi)-\overline{\mu(\phi)})\nabla\phi\right)-A^{-1}\ddiv\left(\overline{\mu(\phi)}\nabla\phi\right)\non\\
&=&A^{-1}S-A^{-1}\ddiv\left((\mu(\phi)-\overline{\mu(\phi)})\nabla\phi\right)-\overline{\mu(\phi)}A^{-1}\ddiv\left(\nabla(\phi-\overline{\phi})\right)\non\\
&=&A^{-1}S-A^{-1}\ddiv\left((\mu(\phi)-\overline{\mu(\phi)})\nabla\phi\right)+\overline{\mu(\phi)}(\phi-\overline{\phi}).
\label{p2}
 \eea
  Applying the Sobolev embeddings $L^{\frac{6}{5}}(\Omega)\hookrightarrow (H^{1}(\Omega))'$, $H^1 \hookrightarrow L^6$ ($d=3$) and H\"{o}lder's inequality, we obtain that
\bea
\|p\|&\leq&\|A^{-1}S\|+\|A^{-1}\ddiv\left((\mu(\phi)-\overline{\mu(\phi)})\nabla\phi\right)\|
+\left|\overline{\mu(\phi)}\right|\|\phi-\overline{\phi}\|\non\\
&\leq&C(\|S\|+\|(\mu-\overline{\mu})\nabla\phi\|_{(H^{1})'})+\left|\overline{\mu(\phi)}\right|\|\phi-\overline{\phi}\|\non\\
&\leq&C(\|S\|+\|(\mu-\overline{\mu})\nabla\phi\|_{L^{\frac{6}{5}}})+\left|\overline{\mu(\phi)}\right|\|\phi-\overline{\phi}\|\non\\
&\leq&C\|S\|+C\|\mu-\overline{\mu}\|_{L^6}\|\nabla\phi\|_{L^\frac32}+\left|\overline{\mu(\phi)}\right|\|\phi-\overline{\phi}\|\non\\
&\leq&C\|S\|+C\|\mu-\overline{\mu}\|_{H^1}\|\nabla\phi\|_{L^\frac32}+\left|\overline{\mu(\phi)}\right|\|\phi-\overline{\phi}\|,\non
\eea
which together with the Poincar\'{e} inequality yields our conclusion \eqref{p00}.
 \end{proof}

\subsection{Global weak solutions}
 The existence of global weak solutions can be obtained by a suitable Galerkin procedure. We consider the eigenvalue problem $-\Delta w= \lambda w$ subject to the homogeneous Neumann boundary condition $\partial_{\nu}w=0$. It is well known that there exist two sequences $\{\lambda_n\}_{n=1,2,...}$ and $\{w_n\}_{n=1,2,...}$ such that, for every
$n\geq  1$, $\lambda_n\geq 0$ is an eigenvalue and $w_n\neq 0$ is a corresponding eigenfunction, the sequence ${\lambda_n}$ is
nondecreasing, tending to infinity as $n\to +\infty$, and the sequence $\{w_n\}$ is orthonormal and complete in $L^2(\Omega)$. We notice that $\lambda = 0$ is an eigenvalue, whence $\lambda_1 = 0$, and that any non-zero constant is an eigenfunction (i.e., $w_1=1$). For every $i > 1$, $w_i$ cannot be a constant and $\int_\Omega w_i dx =0$, whence $\lambda_i=\int_\Omega |\nabla w_i|^2dx>0$. Moreover, as $w_1=1$ is a constant
and $\{w_n\}$ is orthonormal in $L^2(\Omega)$, we easily deduce that
$ A^{-1} w_i=\lambda_i^{-1} w_i$ for every $i>1$.

For any $n\geq 1$, we introduce the finite-dimensional space $W_n={\rm span}\{w_1,...,w_n\}$ and $\Pi_n$ the orthogonal projection on $W_n$. Then we consider the Galerkin approximate problem $(P_n)$:

 Set
\beq\nonumber\phi_n(t,x)=\sum_{i=1}^{n}g_{ni}(t)w_i(x)\eeq which satisfies the
following approximation equation:
\beq\begin{cases}\label{app1}
\partial_t\phi_n=\Delta\mu_n+\Pi_n(S-\ddiv(\ub_n\phi_n)),\\
\mu_n=-\Delta \phi_n+\Pi_n f(\phi_n),\\
\phi_n(\tau)=\Pi_n \phi_\tau,
\end{cases}
\eeq
where $f(\phi_n)=\phi_n^3-\phi_n$ and
\beq
\label{app11}
\ub_n=-\nabla
p_n+\mu_n\nabla\phi_n.
\eeq
Here, $p_n$ satisfies a Poisson equation with homogenous Neumann boundary condition:
\beq
 \begin{cases}-\Delta p_n=S-\ddiv{(\mu_n\nabla\phi_n)}, \quad\text{in }
\Omega,\\
\partial_{\nu}p_n=0, \quad\text{on }
\partial\Omega,\end{cases}
\label{pn}
 \eeq
Then $p_n$ is uniquely determinate up to an arbitrary additive function that may only depend on $t$. For the sake of simplicity and without affecting the mathematical analysis, we require that $\int_{\Omega}p_n dx=0$ and thus
$$ p_n=A^{-1}S-A^{-1}\ddiv(\mu_n\nabla\phi_n).$$

 Taking the inner product of \eqref{app1} in $L^2(\Omega)$  with $w_j$, we
infer that $g_{nj}(t)$ satisfies the following ODE system
\beq \label{app2}
\begin{cases}
g_{nj}'+(\lambda_j^2-\lambda_j)g_{nj}+G_j(g)=S_j(t),\;\;j=1,\cdots,n,\\
g_{nj}(\tau)=\xi_j:=(\phi_\tau,w_j)\end{cases}\eeq where \beq\nonumber
G_j(g)=\lambda_j\left((\sum_{i=1}^{n}g_{ni}w_i)^3,w_j\right)+\left(\ddiv(\ub_n\sum_{i=1}^{n}g_{ni}w_i),w_j\right),
 \eeq
and
\beq\nonumber
 S_j(t)=(S,w_j)\in L^{2}(\tau,T).
 \eeq It is easy to verify that the nonlinearity $G_j$ is locally lipschitz in $g=(g_{n1}, \cdots, g_{nn})$ and as a consequence there exists $T_n\in(\tau, T)$ depending on $|\xi_j|$ such that \eqref{app2} has a unique local solution
$g_{nj}(t)\in C[\tau, T_n].$

In what follows, we derive some a priori estimates on the approximate solutions that are valid in both 2D and 3D.

First, integrating \eqref{app1} over $\Omega\times[\tau, T]$, it is easy to find that
\beq\label{m0}
\int_{\Omega}\phi_n(t)dx=\int_{\Omega}\phi_n(\tau)dx=\int_\Omega \phi_\tau dx, \quad \forall t\in[\tau, T].
\eeq
Multiplying the equation \eqref{app1} by $\mu_n$ and integrating by parts, we get
\bea
&& \frac{d}{dt}\int_{\Omega}\left(\frac{1}{2}|\nabla\phi_n|^2+f(\phi_n)\right)dx+\|\nabla
\mu_n\|^2\non\\
&=&\int_{\Omega}S\mu_n(1-\phi_n)
dx-\int_{\Omega}(\ub_n\cdot\nabla\phi_n)\mu_n dx.\label{app0}
\eea
Taking $L^2$-inner product of \eqref{app11} with $\ub_n$, using integration by parts, we obtain that
\beq
 \|\ub_n\|^2=\int_{\Omega}(-\nabla
p_n+\mu_n\nabla\phi_n)\cdot \ub_ndx=\int_{\Omega}p_n S+(\mu_n\nabla\phi_n)\cdot \ub_n dx.\non
  \eeq
  Summing it with \eqref{app0}, using \eqref{p2} for $p_n$, H\"{o}lder's inequality and Poincar\'{e}'s inequality, we deduce that
\begin{eqnarray}
&&\frac{d}{dt}\int_{\Omega}\left(\frac{1}{2}|\nabla\phi_n|^2+f(\phi_n)\right)dx+\|\nabla
\mu_n\|^2+\|\ub_n\|^2\nonumber\\
&=&\int_{\Omega}S\mu_n(1-\phi_n)dx+\int_{\Omega}p_n S dx\nonumber\\
&=&\int_{\Omega}S(\mu_n-\overline{\mu_n})(1-\phi_n)
dx-\overline{\mu_n}\int_{\Omega} S\phi_n dx\nonumber\\
&&+\int_{\Omega}S\left(A^{-1}S-A^{-1}\ddiv\left((\mu_n-\overline{\mu_n})\nabla\phi_n\right)+\overline{\mu_n}(\phi_n
-\overline{\phi_n})\right) dx\nonumber\\
&=&\int_{\Omega}S(\mu_n-\overline{\mu_n})(1-\phi_n)dx+\int_{\Omega}S\left(A^{-1}S-A^{-1}\ddiv\left((\mu_n-\overline{\mu_n})\nabla\phi_n\right)\right)dx\non\\
&\leq&\|S\|\|\mu_n-\overline{\mu_n}\|+\|S\|_{L^{\frac{3}{2}}}\|\mu_n-\overline{\mu_n}\|_{L^6}\|\phi_n\|_{L^6}\nonumber\\
&&+\|S\|(\|A^{-1}S\|+\|A^{-1}\ddiv\left((\mu_n-\overline{\mu_n})\nabla\phi_n\right)\|)\nonumber\\
&\leq& C\|S\|\|\nabla\mu_n\|(1+\|\phi_n\|_{H^1})+C\|S\|\left(\|S\|+\|\nabla\mu_n\|\|\nabla\phi_n\|_{L^\frac32}\right).
\label{est1a0}
\end{eqnarray}
Thanks to Young's inequality and Poincar\'{e}'s inequality, it holds
\beq\label{poi0}\|\phi_n\|^2_{H^1}=\|\nabla\phi_n\|^2+\|\phi_n\|^2\leq
C\left(\frac{1}{2}\|\nabla\phi_n\|^2+\int_{\Omega}f(\phi_n)dx+1\right).
\eeq
Denote
\beq\nonumber E_0(\phi_n)=\frac{1}{2}\|\nabla\phi_n\|^2+\int_{\Omega}f(\phi_n)dx+1,\eeq we infer from  \eqref{est1a0}, \eqref{poi0} and Young's inequality that
\beq
\label{est1a}
\frac{d}{dt}E_0(\phi_n)+\|\nabla \mu_n\|^2+\|\ub_n\|^2
\leq
\frac{1}{2}\|\nabla\mu_n\|^2+C\|S\|^2E_0(\phi_n).
\eeq

Applying the Gronwall inequality, we obtain that
\beq\label{est1}
\int_{\Omega}\left(\frac{1}{2}|\nabla\phi_n|^2+f(\phi_n)\right)(t)dx+\int_{\tau}^{T}\|\nabla
\mu_n\|^2dt+\int_{\tau}^{T}\|\ub_n\|^2dt
\leq C
\eeq where $C$ depends on  $\|\phi_\tau\|_{H^1}$, $\Omega$ and $\|S\|_{L^{2}(\tau,T;L^{2})}$ but not $T_n$ and $n$. This entails
that
\beq\|\phi_n(t)\|_{H^1}^2=\|(-\Delta+I)^{\frac{1}{2}}\phi_n\|^2=\sum_{i=1}^{n}(1+\lambda_i)g_{ni}^2(t)\leq C\qquad\text{for}\;\tau\leq t\leq T.
\label{est1ah}
\eeq
Hence the local
solution $\phi_n$ can be extended to $[\tau,T]$ for any fixed $T>\tau$.

The estimate \eqref{est1} indicates that  $\ub_n$ is uniformly bounded in $L^2(\tau,T;L^2(\Omega))$. Since
 \beq \left|\int_\Omega \mu_n dx\right|=\left|\int_\Omega f(\phi_n) dx\right|\leq C(\|\phi_n\|_{L^1}+\|\phi_n\|^3_{L^3})\leq C,\eeq
 it follows from \eqref{est1} and the Poincar\'e inequality that $\mu_n$ is uniformly bounded in $L^2(\tau,T; H^1(\Omega))$. Furthermore, by the Gagliardo-Nirenburg inequality ($d=3$), we have
\bea
\|\nabla\Delta\phi_n\|^2&\leq& C\left(\|\nabla\mu_n\|^2+\int_{\Omega}\phi_n^4|\nabla\phi_n|^2dx+\|\nabla\phi_n\|^2\right)\non\\
&\leq&C(1+\|\nabla\mu_n\|^2+\|\phi_n\|_{L^{\infty}}^4)\non\\
&\leq& C(1+\|\nabla\mu_n\|^2+\|\phi_n\|_{L^6}^3\|\nabla\Delta\phi_n\|+\|\phi_n\|_{L^6}^4)\non\\
&\leq&\frac{1}{2}\|\nabla\Delta\phi_n\|^2+C(1+\|\nabla\mu_n\|^2),\non
\eea
which yields that
\beq\nonumber
\int_{\tau}^{T}\|\nabla\Delta\phi_n\|^2dt\leq C.
\eeq
As a consequence, we obtain that
$\phi_n$ is uniformly bounded in $L^{\infty}(\tau,T;H^1(\Omega))$ and also in $L^2(\tau,T;H^3(\Omega))$. By the following interpolation inequality ($d=3$)
\beq\nonumber
 \|\phi_n\|_{L^\infty}\leq C \|\phi_n\|_{L^6}^{\frac{3}{4}}\|\nabla\Delta\phi_n\|^{\frac{1}{4}}+C\|\phi_n\|_{L^6},
\eeq
 it holds that for any $\varphi\in L^{\frac{8}{3}}(\tau,T;H^1(\Omega))$,
\bea
\left|\int_{\tau}^T\int_{\Omega}\ddiv (\ub_n\phi_n)\varphi dxdt\right|
&\leq&\int_{\tau}^T\|\ub_n\|\|\phi_n\|_{L^\infty}\|\nabla \varphi\|dt\non\\
&\leq&\left(\int_{\tau}^T\|\ub_n\|^2dt\right)^{\frac12}\left(\int_{\tau}^T\|\phi_n\|_{L^\infty}^8dt\right)^{\frac18}
\left(\int_{\tau}^T\|\varphi\|_{H^1}^{\frac83}dt\right)^{\frac38}\non\\
&\leq&C.\non
\eea
Therefore, we have
\beq\nonumber
\ddiv(\ub_n\phi_n)\in L^{\frac{8}{5}}(\tau,T;(H^1(\Omega))'),
\eeq
which further implies that
\beq\nonumber\partial_t\phi_n\in L^{\frac{8}{5}}(\tau,T;(H^1(\Omega))')\eeq  is
uniformly bounded.

By the interpolation inequality ($d=3$)
\beq\label{L3}
 \|\nabla \phi_n\|_{L^3}\leq C \|\nabla \phi_n\|^{\frac{3}{4}}\|\nabla\Delta\phi_n\|^{\frac{1}{4}}+C\|\nabla \phi_n\|,
\eeq we have for any $\mathbf{v}\in L^\frac83(\tau, T; \mathbf{L}^2(\Omega))$, it holds
\bea
\left|\int_\tau^T (\mu_n\nabla \phi_n)\cdot\mathbf{v}dt\right|&\leq& \int_\tau^T\|\mu_n\|_{L^6}\|\nabla \phi_n\|_{L^3}\|\mathbf{v}\|dt\non\\
&\leq& C \left(\int_\tau^T \|\mu_n\|_{H^1}^2 dt\right)^\frac12\left(\int_\tau^T\|\nabla\phi_n\|_{L^3}^8dt\right)^\frac18\left(\int_\tau^T\|\mathbf{v}\|^\frac83dt\right)^\frac38\non\\
&\leq& C.
\eea
As a consequence, $\mu_n\nabla\phi_n\in L^{\frac{8}{5}}(\tau,T;\mathbf{L}^2(\Omega))$ and hence we have $\nabla p_n\in  L^{\frac{8}{5}}(\tau,T;\mathbf{L}^2(\Omega))$.

The above uniform estimates are enough to pass to the limit $n\to+\infty$ in the Galerkin scheme by standard compactness theorems to obtain the existence of global weak solutions to the system \eqref{1}--\eqref{ini}. The details are omitted here. One may  refer to \cite{WW2012, BCG} for detailed argument for the simpler case $S=0$.

\subsection{Local strong solutions}

Now we proceed to prove the existence of local strong solutions. For this propose, we derive some higher order a priori estimates for the approximation solutions.

Testing \eqref{app1} by $\Delta^2\phi_n$ and
using integration by parts, we obtain that
\bea
&& \frac{1}{2}\frac{d}{dt}\|\Delta\phi_n\|^2+\|\Delta^2\phi_n\|^2\non\\
&=&
\int_{\Omega}\Delta(\phi_n^3-\phi_n)\Delta^2\phi_ndx+\int_{\Omega}S(1-\phi_n)\Delta^2\phi_ndx
-\int_{\Omega}\ub_n\cdot\nabla\phi_n\Delta^2\phi_ndx\non\\
&\leq&
\frac{1}{4}\|\Delta^2\phi_n\|^2+3\int_{\Omega}\left(|\Delta(\phi_n^3-\phi_n)|^2+S^2(1-\phi_n)^2+|\ub_n|^2|\nabla\phi_n|^2\right)dx,
\label{app3}
\eea
By the three dimensional Agmon's inequality $\|\phi_n\|_{L^{\infty}}\leq C\|\phi_n\|_{H^1}^{\frac12}\|\phi_n\|_{H^2}^{\frac12}$
and the estimate \eqref{est1ah}, we can deduce that
\bea
&&\int_{\Omega}\left|\Delta(\phi_n^3-\phi_n)\right|^2dx\non\\
&\leq&C\int_{\Omega}\left(\phi_n^2|\nabla\phi_n|^4+\phi_n^4|\Delta\phi_n|^2+|\Delta\phi_n|^2\right)dx\non\\
&\leq&C\left(\|\phi_n\|_{L^6}^2\|\nabla\phi_n\|_{L^{6}}^4+\|\phi_n\|_{L^{\infty}}^4\|\Delta\phi_n\|^2+\|\Delta\phi_n\|^2\right)\non\\
&\leq& C(\|\Delta\phi_n\|^2+\|\Delta\phi_n\|^4+1),
\label{i1b}
\eea
and
\beq
\int_\Omega S^2(1-\phi_n)^2 dx\leq (1+\|\phi_n\|_{L^\infty})^2\|S\|^2 \leq C(1+\|\Delta\phi_n\|)\|S\|^2.\label{ubn2}
\eeq
For the third term on the right-hand side of \eqref{app3}, we have
 \bea \int_{\Omega}|\ub_n|^2|\nabla\phi_n|^2dx
&\leq&C\int_{\Omega}\left(|\nabla p_n|^2|\nabla\phi_n|^2+|\mu_n|^2|\nabla\phi_n|^4\right)dx\non\\
&\leq& C\|\nabla\phi_n\|_{L^\infty}^2\|\nabla p_n\|^2+C\|\nabla\phi_n\|_{L^\infty}^4\|\mu_n\|^2.\label{ubn1}
\eea
Using  the estimate \eqref{est1ah}, \eqref{L3} together with Agmon's inequality for $\nabla \phi_n$
\beq \|\nabla \phi_n\|_{L^\infty}\leq C\|\phi_n\|_{H^2}^\frac12\|\phi_n\|_{H^3}^\frac12\non\eeq
and the fact $$\|\nabla p_n\|^2= \int_\Omega \left(S p_n +(\mu_n\nabla\phi_n)\cdot \nabla p_n\right) dx$$
 we have
\bea
\|\nabla\phi_n\|_{L^\infty}^2\|\nabla p_n\|^2
&\leq& C\|\nabla\phi_n\|_{L^\infty}^2(\|S\|_{\dot{H}^{-1}}^2+\|\mu_n\nabla\phi_n\|^2)\non\\
&\leq& C\|\nabla\phi_n\|_{L^\infty}^2\|S\|_{\dot{H}^{-1}}^2+C\|\nabla\phi_n\|_{L^\infty}^4\|\mu_n\|^2,\label{i31b}
\eea where
\bea
\|\nabla\phi_n\|_{L^\infty}^4\|\mu_n\|^2
&\leq&C(1+\|\nabla\phi_n\|_{H^1}^2\|\nabla\phi_n\|_{H^2}^2)(1+\|\Delta \phi_n\|^2)\non\\
&\leq& C(1+\|\Delta \phi_n\|^2\|\nabla\Delta\phi_n\|^2+\|\Delta\phi_n\|^2+\|\nabla\Delta\phi_n\|^2)(1+\|\Delta\phi_n\|^2)\non\\
&\leq& \frac18\| \Delta^2\phi_n\|^2+C(\|\Delta \phi_n\|^{10}+1),\label{i32b}
\eea
 and
 \bea \|\nabla\phi_n\|_{L^\infty}^2\|S\|_{\dot{H}^{-1}}^2
 &\leq&C(1+\|\Delta \phi_n\|\|\nabla\Delta\phi_n\|+\|\Delta\phi_n\|+\|\nabla\Delta\phi_n\|)\|S\|_{\dot{H}^{-1}}^2\non\\
 &\leq&\frac18\| \Delta^2\phi_n\|^2+C(\|\Delta \phi_n\|^{2}+1)\eea
As a consequence, we obtain from \eqref{app3}--\eqref{i32b} that
\beq
\label{app4}\frac{d}{dt}\|\Delta\phi_n\|^2+\|\Delta^2 \phi_n\|^2\leq C\left(\|\Delta\phi_n\|^{10}+1\right).
\eeq
Letting $y_n(t)=\|\Delta\phi_n\|^2+1$, we have
\beq y_n'(t)\leq C_0y_n^{5}(t)\eeq
with the constant $C_0$ is independent of $t$. Solving this inequality implies that
\beq
\nonumber y_n(t)\leq \frac{y_n(\tau)}{(1-4C_0y_n^{4}(\tau)t)^{\frac{1}{4}}}, \qquad \forall\; \tau\leq t\leq\min\left\{\frac{1}{4C_0y_n^{4}(\tau)}, T\right\}:=T_n.
\eeq
Noticing that
\beq\nonumber y_n(\tau)\leq y(\tau)=\|\Delta\phi_\tau\|^2+1,\eeq we get
\beq\nonumber y_n(t)\leq 2^{-\frac{1}{4}}(\|\Delta\phi_\tau\|^2+1),\qquad\text{ whenever   }\tau\leq t\leq\min\left\{\frac{1}{8C_0(\|\Delta\phi_\tau\|^2+1)^{4}}, T\right\}:=T^*.\eeq
As a result, for any $t\in[\tau,T^*]$, the following estimate holds
\beq\label{app5}\|\phi_n(t)\|_{H^2}^2+\int_{\tau}^{T^*}\|\phi_n(t)\|_{H^4}^2dt\leq C.\eeq
The above estimate together with \eqref{ubn2}--\eqref{i32b} yields
\beq \int_\tau^{T^*} \|\ddiv(\ub_n\phi_n)\|^2dt\leq C.\non\eeq
Besides,
\beq
\int_\tau^{T^*}\|\mu_n\|_{H^2}^2dt\leq C\int_\tau^{T^*}(\|\Delta^2\phi_n\|^2+\|\phi_n\|_{H^2}^2+\|\phi_n\|_{H^2}^6)dt\leq C.\label{muh2}
\eeq
As a consequence, we also have
\beq \int_\tau^{T^*}\|\partial_t\phi_n\|^2 dt\leq C\eeq
and
\bea
\int_\tau^{T^*} \|p\|_{H^2}^2dt
&\leq& C\int_\tau^{T^*} \left(\|S\|^2+\|\ddiv(\mu_n\nabla\phi_n)\|^2\right)dt\non\\
&\leq& C+\int_\tau^{T^*} (\|\nabla \mu_n\|_{L^3}^2\|\nabla\phi_n\|_{L^6}^2+\|\mu_n\|_{L^\infty}^2\|\phi_n\|_{H^2}^2)dt\non\\
&\leq& C.\label{ph2}
\eea
Finally, from \eqref{muh2} and \eqref{ph2} we can easily derive that
\beq
\int_\tau^{T^*} \|\ub_n\|_{H^1}^2dt\leq C.
\eeq

Combining the above estimates together, we are able to prove the existence of local strong solution to the system \eqref{1}--\eqref{ini} by the same argument as in \cite{LTZ}. Moreover, arguing exactly as in \cite[Section 6]{LTZ}, we can obtain the uniqueness of strong solutions. This completes the proof of Theorem \ref{Th1}.

\section{Global Strong Solution in 2D}
In this section, we focus on the study of the CHD system \eqref{1}--\eqref{ini} in the 2D case and prove  Theorem \ref{Th2}. Differently from the 3D case, the strong solution exists globally under weak assumption on the external source term $S$. Moreover, it defines a family of closed processes $\{U(t, \tau)\}_{t\geq \tau}$ in the space $H_N^2(\Omega)$.

\subsection{Existence}
We show that under a slightly weak assumption on $S$ than in Theorem \ref{Th1}(ii), one can actually prove the existence of global strong solution to the system \eqref{1}--\eqref{ini}. Based on the Galerkin scheme described before, we only need to obtain proper global-in-time \textit{a priori} estimates. For the sake of simplicity, below we shall just perform formal estimates for smooth solutions (i.e., drop the subscript '$n$'), which can be rigorously justified by the Galerkin approximation in previous section.

\begin{lemma}\label{regww}
Suppose that $d=2$ and $S\in L^2(\tau, T; \dot L^2(\Omega))$. Let $(\phi, \ub, p)$ be a smooth solution to problem \eqref{1}--\eqref{ini}. Then the following estimates hold
\beq \label{regw}
\|\Delta\phi(t)\|^2\leq C_1\left(1+\frac{1}{t-\tau}\right),\quad \forall\, t\in (\tau, T],
\eeq
and
\beq\label{h2es}
\|\Delta\phi(t)\|^2 +\int_\tau^T \|\Delta^2\phi(t)\|^2dt \leq C_2,\quad \forall\,t\in[\tau, T]
\eeq
where the constant $C_1$ depends on
 $\|\phi_\tau\|_{H^1}$, $\Omega$ and $\|S\|_{L^{2}(\tau,T;L^{2})}$, while the constant $C_2$ depends on
 $\|\phi_\tau\|_{H^2}$, $\Omega$ and $\|S\|_{L^{2}(\tau,T;L^{2})}$.
\end{lemma}
\begin{proof}
Similar to \eqref{est1}, we have the following estimate
\beq\label{esh1}
\sup_{t\in[\tau, T]}\|\phi(t)\|_{H^1}^2+\int_{\tau}^{T}\|\nabla
\mu\|^2dt+\int_{\tau}^{T}\|\ub\|^2dt
\leq C
\eeq where $C$ depends on and $\|\phi_\tau\|_{H^1}$, $\Omega$ and $\|S\|_{L^{2}(\tau,T;L^{2})}$.
Next, it is similar to \eqref{app3} that by testing \eqref{1} by $\Delta^2\phi$ and
using integration by parts, we obtain
\bea
&& \frac{1}{2}\frac{d}{dt}\|\Delta\phi\|^2+\frac34\|\Delta^2\phi\|^2\non\\
&\leq& 3\int_{\Omega}\left(|\Delta(\phi^3-\phi)|^2+S^2(1-\phi)^2+|\ub|^2|\nabla\phi|^2\right)dx,
\label{ap3}
\eea
 Using the two dimensional Agmon's inequality $\|\phi\|_{L^\infty}\leq C \|\phi\|^\frac12\|\phi\|_{H^2}^\frac12$ and the Gagliardo-Nirenberg inequality $\|\nabla \phi\|_{L^4}\leq C\|\nabla \Delta \phi\|^\frac14\|\nabla \phi\|^\frac34+C\|\nabla \phi\|$, we can estimate the first two terms on the right-hand side of \eqref{ap3} as follows:
\bea && 3\int_{\Omega}\left|\Delta(\phi^3-\phi)\right|^2dx\non\\
&\leq& C\int_{\Omega}\left(\phi^2|\nabla\phi|^4+\phi^4|\Delta\phi|^2+|\Delta\phi|^2\right)dx\non\\
&\leq& C\left(\|\phi\|_{L^{\infty}}^2\|\nabla\phi\|_{L^4}^4+\|\phi\|_{L^{\infty}}^4\|\Delta\phi\|^2+\|\Delta\phi\|^2\right)\non\\
&\leq& C(\|\phi\|^2+\|\Delta\phi\|\|\phi\|)(\|\nabla\Delta\phi\|\|\nabla \phi\|^3+\|\nabla \phi\|^4)\non\\
&& +C(\|\Delta\phi\|^2\|\phi\|^2+\|\phi\|^4)\|\Delta\phi\|^2+C\|\Delta\phi\|^2\non\\
&\leq& C\|\phi\|_{H^1}^3(\|\phi\|_{H^1}^2+\|\Delta\phi\|^2)(\|\nabla\Delta\phi\|+\|\phi\|_{H^1})+C\|\Delta\phi\|^2,\label{i1}
\eea
where we have used the interpolation
$\|\Delta \phi\|^2\leq \|\nabla \phi\|\|\nabla \Delta\phi\|$, which is a consequence of the fact that $\phi$ fulfils $\partial_\nu \phi=0$ on the boundary. Besides, it is easy to see that
\beq 3 \int_{\Omega}S^2(1-\phi)^2\leq C\|S\|^2(1+\|\phi\|_{L^{\infty}})^2\leq C\|S\|^2(\|\Delta\phi\|\|\phi\|+\|\phi\|^2).\label{i2}
\eeq
For the third term on the right-hand side of \eqref{ap3},  we deduce from \eqref{p2} that
\bea
&&3\int_{\Omega}|\ub|^2|\nabla\phi|^2dx\non\\
&\leq& C\int_\Omega |\nabla p|^2|\nabla\phi|^2dx+C\|\nabla\phi\|_{L^\infty}^4\|\mu\|^2\non\\
&\leq& C\int_\Omega |\nabla A^{-1}S|^2|\nabla \phi|^2dx+\int_\Omega |\nabla A^{-1}\ddiv(\mu\nabla\phi)|^2|\nabla \phi|^2dx+C\|\nabla\phi\|_{L^\infty}^4\|\mu\|^2\non\\
&\leq& C\|\nabla A^{-1}S\|_{L^4}^2\|\nabla \phi\|_{L^4}^2+C\|\nabla\phi\|_{L^\infty}^4\|\mu\|^2\non\\
&\leq& C\|S\|^2(\|\nabla \phi\|^2+\|\nabla \phi\|\|\Delta\phi\|)\non\\
&& +C(\|\nabla \phi\|^4+\|\nabla \phi\|^2\|\Delta \phi\|^2+\|\nabla \phi\|^2\|\nabla \Delta\phi\|^2)(\|f'(\phi)\|^2+\|\Delta \phi\|^2)\non\\
&\leq& C\|S\|^2(\|\nabla \phi\|^2+\|\Delta\phi\|^2)\non\\
&& +C\|\nabla \phi\|^2(\|\nabla \phi\|^2+\|\nabla \Delta\phi\|^2)(\|\phi\|_{H^1}^6+\|\phi\|_{H^1}^2+\|\Delta \phi\|^2)
\label{i31ba}
\eea
Here we note that the constants $C$ in \eqref{i1}--\eqref{i31ba} depend  only on $\Omega$ and coefficient of the system.

As a consequence, we deduce from  \eqref{ap3}--\eqref{i31ba} and the uniform estimate \eqref{esh1} that
\beq\begin{split}\label{ap4}
\frac{d}{dt}\|\Delta\phi\|^2+\|\Delta^2\phi\|^2\leq Ch(t)\|\Delta\phi\|^2+Ch(t),
\end{split}
\eeq
where $$h(t)=1+\|S\|^2+\|\nabla\Delta\phi\|^2$$
and the constant $C$ in \eqref{ap4} depends on  $\|\phi_\tau\|_{H^1}$, $\Omega$ and $\|S\|_{L^{2}(\tau,T;L^{2})}$.\\
Besides, it easily follows from \eqref{esh1} that
\beq
\sup_{t\in [\tau, T)} \int_{t}^{t+r} h(s) ds\leq r+C, \quad \forall\, r\in (0, \min\{1, T-t\}).
\eeq
Then by the uniform Gronwall inequality \cite[Lemma III.1.1]{TE}, we infer that
\beq \label{est2}
\|\Delta\phi(t+\delta)\|^2\leq C(1+\delta^{-1}),\quad \forall\ t\in [\tau, T),\  \delta\in (0, \min\{1, T-t\}),
\eeq
where the constant $C$ depends on
 $\|\phi_\tau\|_{H^1}$, $\Omega$ and $\|S\|_{L^{2}(\tau,T;L^{2})}$.

On the other hand, by the classical Gronwall inequality, we also infer that
\beq
\|\Delta\phi(t)\|^2 \leq (\|\Delta\phi_\tau\|^2+1)e^{C\int_\tau^T h(s)ds},
\eeq
and then
\beq
\int_\tau^T \|\Delta^2\phi(t)\|^2dt\leq C,
\eeq
where the constant $C$ depends on
 $\|\phi_\tau\|_{H^2}$, $\Omega$ and $\|S\|_{L^{2}(\tau,T;L^{2})}$.
\end{proof}

The existence of global strong solutions to problem \eqref{1}--\eqref{ini} is a direct consequence of the uniform estimates \eqref{h2es} and \eqref{esh1} (see \cite[Section 4]{LTZ} for detailed argument with $S=0$). Thus, the proof is omitted here.

\subsection{Continuous dependence on initial data}
The strong solution to problem \eqref{1}--\eqref{ini} satisfies the following continuous dependence property, which also yields the uniqueness:
\begin{lemma}\label{pconti}
Suppose that $d=2$. Let $(\phi_i,\ub_i,p_i)$ $(i=1,2)$ be the two global strong solutions corresponding to the initial data $\phi_{\tau i}\in H^2_N(\Omega)$. Then for $t\in [\tau, T]$, the following estimate holds:
\beq
\|\phi_1(t)-\phi_2(t)\|_{H^1}^2+\int_\tau ^T(\|\nabla \mu(s)\|^2+\|\ub(s)\|^2)ds \leq C_T \|\phi_{\tau 1}-\phi_{\tau 2}\|^2_{H^1},\label{lipconti}
\eeq
where the constant $C_T$ may depends on $\|\phi_{\tau 1}\|_{H^2}$, $\|\phi_{\tau 2}\|_{H^2}$, $\int_\tau^T\|S\|^2ds$, $\Omega$, $\tau$ and $T$.
\end{lemma}

\begin{proof}
The argument is similar to \cite[Section 6]{LTZ} with minor modifications due to the appearance of the source term $S$. For the convenience of the readers, we sketch the proof here. Let us set $\phi=\phi_1-\phi_2$, $\ub=\ub_1-\ub_2$ and $p=p_1-p_2$. Also denote $\mu_i=-\Delta\phi_i+f(\phi_i)$, $i=1,2$ and $\mu:=\mu_1-\mu_2=-\Delta \phi+f(\phi_1)-f(\phi_2)$.
Then $(\phi, \ub, p)$ solves the system
\beq
 \begin{cases}
 \phi_t+\ddiv(\ub\phi_1+\ub_2\phi)=\Delta \mu,\\
 \ub=-\nabla p+(\mu\nabla\phi_1+\mu_2\nabla\phi),\\
 \ddiv\ub=0,
 \end{cases}
 \label{uni1}
\eeq
subject to boundary and initial conditions
\beq
 \begin{cases}
  \partial_{\nu}\phi=\partial_{\nu}\mu=\ub\cdot\nu=0\quad\text{on }
\partial\Omega, \\
\phi(t,x)|_{t=\tau}=\phi_{\tau 1}-\phi_{\tau 2}.
\end{cases}\non
\eeq
Testing the first equation of \eqref{uni1} by $\phi$, after integration by parts we obtain that
\bea
&&\frac{1}{2}\frac{d}{dt}\|\phi\|^2+\|\Delta\phi\|^2\non\\
&=&\int_\Omega(f'(\phi_1)-f'(\phi_2))\Delta\phi dx
-\frac{1}{2}\int_{\Omega}S\phi^2dx+\int_\Omega \phi_1\ub\cdot \nabla\phi dx \non\\
&:=& I_1+I_2+I_3.
\eea
Using the uniform estimates \eqref{esh1} and Agmon's inequality, the terms $I_1$, $I_3$ can be estimated as in \cite[(6.9)]{LTZ} such that
\bea
I_1&\leq& (1+\|\phi_1^2+\phi_1\phi_2+\phi_2^2\|_{L^\infty})\|\phi\|\|\Delta \phi\|\non\\
&\leq& \frac{1}{4}\|\Delta \phi\|^2+C\|\phi\|^2,
\eea
\beq
I_3\leq  \|\ub\|\|\nabla\phi\|\|\phi_1\|_{L^\infty}\leq \frac18\|\ub\|^2+C\|\nabla \phi\|^2.
\eeq
Concerning $I_2$, we have
\bea
I_2&\leq& \frac12\|S\|\|\phi\|\|\phi\|_{L^\infty}\leq C\|S\|\|\phi\|^\frac32\|\phi\|_{H^2}^\frac12\non\\
&\leq& C\|S\|\|\phi\|^2+C\|S\|\|\phi\|^\frac32\|\Delta\phi\|^\frac12\non\\
&\leq& \frac{1}{4}\|\Delta \phi\|^2+C(\|S\|^2+1)\|\phi\|^2.
\eea
As a consequence, we have
\beq
\frac{d}{dt}\|\phi\|^2+\|\Delta\phi\|^2\leq \frac14\|\ub\|^2+C(\|S\|^2+1)(\|\nabla \phi\|^2+\|\phi\|^2).\label{d1a}
\eeq
Next, testing the first and the second equations of \eqref{uni1} by $\mu$ and $\ub$
 respectively, adding the results together, we obtain that
\bea
&&\frac{d}{dt}\left(\frac{1}{2}\|\nabla\phi\|^2-\frac12\|\phi\|^2+\frac14\int_\Omega \phi^4 dx\right)+\|\nabla\mu\|^2+\|\ub\|^2\non\\
&=&\int_{\Omega}\mu_2 \nabla\phi \cdot \ub dx+\int_{\Omega}\phi\ub_2\cdot\nabla\mu  dx+3\int_\Omega \phi_1\phi_2\phi \phi_t dx,\label{dif1}
\eea
The first two terms on the right-hand side of \eqref{dif1} can be estimated exactly like \cite[(6.6)-(6.7)]{LTZ} that
\bea
&& \int_{\Omega}\mu_2 \nabla\phi \cdot \ub dx+\int_{\Omega}\phi\ub_2\cdot\nabla\mu  dx\non\\
&\leq& \frac18\|\nabla \mu\|^2+\frac18\|\ub\|^2+C(\|\phi_2\|_{H^4}^2+\|\ub_2\|_{H^1}^2)(\|\nabla \phi\|^2+\|\phi\|^2).\label{dif2}
\eea
For the third term, we have
\bea
&& 3\int_\Omega \phi_1\phi_2\phi \phi_t dx\non\\
&=& -3\int_\Omega \ddiv(\ub \phi_1+\ub_2\phi) \phi_1\phi_2\phi dx+3\int_\Omega \phi_1\phi_2\phi\Delta\mu dx\non\\
&=& -3\int_\Omega (\ub \cdot\nabla \phi_1)\phi_1\phi_2\phi dx -3\int_\Omega S \phi_1\phi_2\phi^2 dx-3\int_\Omega (\ub_2\cdot\nabla \phi)\phi_1\phi_2\phi dx\non\\
&& -3\int_\Omega \nabla(\phi_1\phi_2\phi)\cdot\nabla\mu dx\non\\
&\leq& \frac18\|\ub\|^2+C\|\nabla \phi_1\|_{L^4}^2\|\phi_1\|_{L^\infty}^2\|\phi_2\|_{L^\infty}^2\|\phi\|_{L^4}^2+
C\|S\|\|\phi_1\|_{L^\infty}\|\phi_2\|_{L^\infty}\|\phi\|_{L^4}^2
\non\\
&&+C\|\ub_2\|_{L^4}\|\phi_1\|_{L^\infty}\|\phi_2\|_{L^\infty}\|\nabla \phi\|\|\phi\|_{L^4}+\frac18\|\nabla \mu\|^2+ C\|\phi_1\|_{L^\infty}^2\|\phi_2\|_{L^\infty}^2\|\nabla \phi\|^2\non\\
&&+C\|\nabla \phi_1\|_{L^\infty}^2\|\phi_2\|_{L^\infty}^2\|\phi\|^2+C\|\phi_1\|_{L^\infty}^2\|\nabla \phi_2\|_{L^\infty}^2\|\phi\|^2\non\\
&\leq& \frac18\|\ub\|^2+\frac18\|\nabla \mu\|^2\non\\
&&+C(\|\phi_2\|_{H^3}^2+\|\phi_1\|_{H^3}^2+\|\ub_2\|_{H^1}^2+\|S\|^2+1)(\|\nabla \phi\|^2+\|\phi\|^2).\label{dif3}
\eea
As a consequence, we infer from \eqref{dif1}--\eqref{dif3} that
\bea
&&\frac{d}{dt}\left(\frac{1}{2}\|\nabla\phi\|^2-\frac12\|\phi\|^2+\frac14\int_\Omega \phi^4 dx\right)+\frac34\|\nabla \mu\|^2+\frac34\|\ub\|^2\non\\
&\leq& C(\|\phi_2\|_{H^4}^2+\|\phi_1\|_{H^3}^2+\|\ub_2\|_{H^1}^2+\|S\|^2+1)(\|\nabla \phi\|^2+\|\phi\|^2).\label{d2a}
\eea
Adding \eqref{d1a} with \eqref{d2a}, we obtain that
\bea
&&
\frac{d}{dt}\left(\frac{1}{2}\|\nabla\phi\|^2+\frac12\|\phi\|^2+\frac14\int_\Omega \phi^4 dx\right)+\frac34\|\nabla \mu\|^2+\frac12\|\ub\|^2\non\\
&\leq& Ch(t)\left(\|\nabla \phi\|^2+\|\phi\|^2+\frac12\int_\Omega \phi^4 dx\right),
\eea
where $$h(t)=\|\phi_2\|_{H^4}^2+\|\phi_1\|_{H^3}^2+\|\ub_2\|_{H^1}^2+\|S\|^2+1.$$
Due to \eqref{h2es},
\beq\nonumber\int_{\tau}^{t}h(s)ds\leq C,\quad\forall t\in [\tau, T],\eeq
where the constant $C$ depends on $\|\phi_2(\tau)\|_{H^2}$, $\int_\tau ^T \|S\|^2 ds$, $\tau$ and $T$. Thus by the Gronwall inequality, we deduce that for all $t\in [\tau, T]$
\bea
&&\|\nabla\phi(t)\|^2+\|\phi(t)\|^2+\frac12\int_\Omega \phi^4 dx \non\\
&\leq& \displaystyle{e^{C\int_\tau^T h(s)ds}} \left(\|\nabla(\phi_{\tau 1}-\phi_{\tau 2})\|^2+\|\phi_{\tau 1}-\phi_{\tau 2}\|^2+\frac12\|\phi_{\tau 1}-\phi_{\tau 2}\|_{L^4}^4\right).\non
\eea
Our conclusion  \eqref{lipconti} easily follows from the above estimate. The proof is complete.
\end{proof}

\subsection{Associated process}

Recall the following definition (see \cite{GMR12}, we also refer to \cite{PZ07} for the definition of closed semigroups):
\begin{definition}\label{clodef}
Let $X$ be a metric space. The set class $\{U(t,\tau)\}_{t\geq \tau}$ that $U(t, \tau):X\to X$ is called a {\rm process} on $X$, if (i) $U(\tau, \tau)x=x$ for any $x\in X$; (ii) $U(t, \tau)x=U(t,s)U(s, \tau)x$ for any $\tau\leq s\leq t$ and any $x\in X$.

Moreover, a process $\{U(t,\tau)\}_{t\geq\tau}$
is said to be {\rm closed} on $X$, if for any $\tau\leq t$,
and any sequence $\{x_n\}\in X$ with $x_n\to x\in X$ and $U(t,\tau)x_n\to y\in X$, then $U(t,\tau)x=y$.
\end{definition}

 Then we infer from Lemma \ref{pconti} that
  \begin{proposition}\label{process}
  For any $S\in L^2_{loc}(\mathbb{R}; \dot L^2(\Omega))$, we are able to define a family of closed processes
$\{U(t,\tau)\}_{t\geq \tau}$ on $\mathcal{H}=H^2_N(\Omega)$ as follows:
\begin{equation*}
U(t,\tau)\phi_\tau=\phi(t; \tau, \phi_\tau),\quad \forall\,\phi_\tau\in H^2_N(\Omega), \quad \forall\, \tau\leq t,
\end{equation*}
where $\phi(t)$ is the unique global strong solution to problem \eqref{1}--\eqref{b2}.
 \end{proposition}

\section{Pullback Attractor in 2D}

In this section, we study the long-time dynamics of the family of  processes $\{U(t,\tau)\}_{t\geq \tau}$ defined by the global strong solution to CHD problem \eqref{1}--\eqref{ini} in terms of the \emph{pullback attractor}. To this end, we first introduce some basic definitions and abstract results about pullback attractors for closed processes adopted from \cite{GMR12} (cf. \cite{Me10} for the case of closed cocycles).

\subsection{Preliminaries}
Consider a metric space $(X, \mathrm{d}_X)$. We denote by $\operatorname{dist}_X(B_1, B_2)$ the Hausdorff semi-distance in $X$ between two sets $B_1, B_2\subset X$ defined as $\operatorname{dist}_X(B_1, B_2)=\sup_{x\in B_1}\inf_{y\in B_2}\mathrm{d}_X(x, y)$.  $\mathcal{P}(X)$ stands for the family of all nonempty subsets of $X$. Let $\mathcal{D}$ be a nonempty class of families parameterized in time $\hat{D}=\{D(t):t\in\mathbb{R}\}\subset \mathcal{P}(X)$. The class $\mathcal{D}$ is called a \textit{universe} in $\mathcal{P}(X)$ (see \cite{MR09}).

We recall now some definitions that will be useful in the  subsequent analysis (see e.g., \cite{ct, GMR12}):

\begin{definition}\label{absor}
A family of nonempty sets $\hat{D}_0=\{D_0(t): t\in \mathbb{R}\}\subset\mathcal{P}(X)$  is said to be {\rm pullback $\mathcal{D}$-absorbing}
for the process $\{U(t,\tau)\}_{t\geq\tau}$, if for any $\hat{D}\in\mathcal{D}$
and any $t\in\mathbb{R}$, there exists a $\tau_0(t,\hat{D})\leq t$ such that
 $U(t,\tau)D(\tau)\subset D_0(t)$ for any $\tau\leq\tau_0(t,\hat{D})$.
\end{definition}

\begin{definition}\label{comp}
The process $\{U(t,\tau)\}_{t\geq\tau}$ is said to be {\rm pullback
$\mathcal{D}$-asymptotically compact}, if for any $t\in\mathbb{R}$ and any
$\hat{D}\in\mathcal{D}$, any sequence $\tau_n\to -\infty$ and any
sequence $x_n\in D(\tau_n)$, the sequence $\{U(t,\tau_n)x_n\}_{n=1}^{\infty}$
is relatively compact in $X$.
\end{definition}

\begin{definition}\label{pull}
A family $\mathcal{A}_\mathcal{D}=\{A_\mathcal{D}(t):t\in\mathbb{R}\}$ of nonempty subsets of
$X$ is said to be {\rm a pullback $\mathcal{D}$-attractor} for the process
$\{U(t,\tau)\}_{t\geq\tau}$ in $X$, if
\begin{itemize}
\item [(i)] $A_\mathcal{D}(t)$ is  compact in $X$ for any $t\in\mathbb{R}$,
\item [(ii)] ${\mathcal{A}}_\mathcal{D}$ is invariant, i.e., $U(t,\tau)A_\mathcal{D}(\tau)=A_\mathcal{D}(t)$
 for any $\tau\leq t$,
\item [(iii)] ${\mathcal{A}}_\mathcal{D}$ is pullback $\mathcal{D}$-attracting, i.e., for any $t\in\mathbb{R}$ and any $\hat{D}=\{D(t):t\in\mathbb{R}\}\in\mathcal{D}$, it holds
\begin{equation*}
\lim_{\tau\to -\infty}\operatorname{dist}_X(U(t,\tau)D(\tau),A_\mathcal{D}(t))=0.
\end{equation*}

\end{itemize}
\end{definition}

 The following abstract result on the existence of minimal pullback attractors for closed processes is proved in \cite{GMR12} (see also \cite{Me10} for the case of closed cocycles):

\begin{lemma} \label{34}
Consider  a closed process $\{U(t,\tau)\}_{t\geq\tau}$ in $X$. Let $\mathcal{D}$ be a universe in $\mathcal{P}(X)$. If the following conditions are satisfied:
\begin{itemize}
\item [(1)] there exists a family $\hat D_0=\{D_0(t): t\in \mathbb{R}\}\subset \mathcal{P}(X)$  such that $\hat D_0$ is pullback $\mathcal{D}$-absorbing for $\{U(t,\tau)\}_{t\geq\tau}$,
\item [(2)] $\{U(t,\tau)\}_{t\geq\tau}$ is pullback $\mathcal{D}$-asymptotically compact,
\end{itemize}
then there exists a minimal pullback $\mathcal{D}$-attractor
${\mathcal{A}}_\mathcal{D}=\{A_\mathcal{D}(t):t\in\mathbb{R}\}$ in $X$ given by
$$ A_\mathcal{D}(t)=\overline{\bigcup_{\hat D\in \mathcal{D}}\Lambda(\hat D, t)}^X,$$
where
\begin{equation*}
\Lambda(\hat D, t)= \bigcap_{s\leq t}\overline{\bigcup_{\tau\leq s}U(t,\tau)D(\tau)}^X, \quad \hat{D}\in \mathcal{D}.
\end{equation*}
\end{lemma}
\begin{remark}\label{pullre}
(i) Such a family ${\mathcal{A}}_\mathcal{D}$ is  {\rm minimal} in the sense that if $\hat{C}=\{C(t):t\in\mathbb{R}\}\subset\mathcal{P}(X)$ is a family of closed subsets such that for any $\hat{D}=\{D(t):t\in\mathbb{R}\}\in\mathcal{D}$,  $$\lim_{\tau\to -\infty}\operatorname{dist}_X(U(t,\tau)D(\tau),C(t))=0,$$ then ${\mathcal{A}}_\mathcal{D}(t)\subset C(t)$.

(ii) In the definition above, $\hat D_0$ does not necessarily belong to the class $\mathcal{D}$. Furthermore, if $\hat D_0\in \mathcal{D}$, then we have $A_\mathcal{D}(t)=\Lambda(\hat D_0, t)\subset \overline{D_0(t)}^X$.
\end{remark}

\subsection{Existence of pullback $\mathcal{D}_F^{\mathcal{H}_M}$-absorbing sets}

Since our system \eqref{1}--\eqref{3} preserves the spatial average of $\phi$ (see \eqref{mm}), it seems impossible
to construct a suitable absorbing set for the process $\{U(t,\tau)\}_{t\geq \tau}$ on the whole space $\mathcal{H}:=H^2_N(\Omega)$. Instead, we shall study the dynamics of problem \eqref{1}--\eqref{ini} confined on the phase space $\mathcal{H}_M$ (see \eqref{H0} for its definition).

For the sake of simplicity, in the subsequent text, we denote by $\mathcal{D}_F^{\mathcal{H}_M}$ the class of families $\hat D=\{D(t)=D: t\in \mathbb{R}\}$ with $D$ being a nonempty fixed bounded subset of $\mathcal{H}_M$ (i.e., $\hat D\subset \mathcal{P}(\mathcal{H}_M)$ and $D$ is parameterized in time but constant for all $t\in \mathbb{R}$, see \cite{CDF97}). Then $\mathcal{D}_F^{\mathcal{H}_M}$ is the universe we shall work on.

First, we prove the existence of a pullback $\mathcal{D}_F^{\mathcal{H}_M}$-absorbing family of sets for the process $\{U(\tau, t)\}_{t\geq \tau}$:

\begin{proposition}\label{plab} Let $d=2$. Suppose that $S\in L^2_b(\mathbb{R}; \dot L^2(\Omega))$.
Then there is a family $\hat D_0\subset \mathcal{D}_F^{\mathcal{H}_M}$ that is pullback $\mathcal{D}_F^{\mathcal{H}_M}$-absorbing for the processes $\{U(t,\tau)\}_{t\geq\tau}$ associated with problem \eqref{1}--\eqref{ini}.
\end{proposition}
\begin{proof} In the subsequent proof, $C$, $C_i$ denote constants that may depend on $\Omega$, $M$, but are independent of the initial datum for $\phi$. $\mathcal{Q}_i(\cdot)$ stand for certain monotone increasing functions.

Multiplying  \eqref{1} by $\mu$ and \eqref{2} by  $\ub$, integrating over $\Omega$ then adding the resultants together (comparing with \eqref{est1a0} for the approximate solutions), we deduce from the H\"older inequality and the Poincar\'e inequality that
\bea
&&\frac{d}{dt} E(\phi) +\|\nabla
\mu\|^2+\|\ub\|^2\non\\
&=&\int_{\Omega}S(\mu-\overline{\mu})(1-\phi)
dx+\int_{\Omega} S\left(A^{-1}S-A^{-1}\left(\ddiv\left((\mu-\overline{\mu})\nabla\phi\right)\right)\right)dx\non\\
&\leq&\|S\|\|\mu-\overline{\mu}\|+
\|S\|\|\mu-\overline{\mu}\|_{L^4}\|\phi\|_{L^4}\non\\
&& +\|S\|\left(\|A^{-1}S\|+\|A^{-1}\left(\ddiv\left((\mu-\overline{\mu})\nabla\phi\right)\right)\|\right)\non\\
&\leq&C\|S\|\|\nabla\mu\|(1+\|\phi\|_{L^4})+C\|S\|\left(\|A^{-1}S\|
+\|\nabla\mu\|\|\nabla\phi\|_{L^{\frac32}}\right),
\label{est1a1}
\eea
where $$E(\phi)=\int_{\Omega}\left(\frac{1}{2}|\nabla\phi|^2+f(\phi)\right)dx.$$
By the two dimensional Gagliardo-Nirenberg inequality and Young's inequality, we have
\bea
&& C\|S\|\|\nabla\mu\|(1+\|\phi\|_{L^4})+C\|S\|\left(\|A^{-1}S\|
+\|\nabla\mu\|\|\nabla\phi\|_{L^{\frac32}}\right)
\non\\
&\leq&\frac{1}{4}\|\nabla\mu\|^2+C\|S\|^2(1+\|\phi\|_{L^4}^2+\|\nabla\phi\|_{L^{\frac32}}^2)\non\\
&\leq&\frac{1}{4}\|\nabla\mu\|^2+C\|S\|^2\left(1+\|\nabla\phi\|^{\frac23}\|\phi\|_{L^4}^{\frac43}+\|\phi\|_{L^4}^2\right)\non\\
&\leq&\frac{1}{4}\|\nabla\mu\|^2+C\|S\|^2\left(1+\|\phi\|_{L^4}^{\frac83}+\|\nabla\phi\|^{\frac43}\right).\label{ge1}
\eea
From estimates \eqref{est1a1}--\eqref{ge1} and  Young's inequality we infer that
\bea
&&
 \frac{d}{dt}\int_{\Omega}\left(\frac{1}{2}|\nabla\phi|^2+f(\phi)\right)dx
 +\frac{1}{2}\|\nabla\mu\|^2+\|\ub\|^2\non\\
 &\leq&
  C_1\|S\|^2\left(1+\|\phi\|_{L^4}^{\frac83}+\|\nabla\phi\|^\frac43\right).
  \label{est1a2}
\eea
Recalling the mass conservation property \eqref{mm}, we rewrite equation \eqref{1} in the following form
\beq
\label{3m}
(\phi-\overline{\phi})_t+\Delta^2(\phi-\overline{\phi})-\Delta (f'(\phi)-\overline{f'(\phi)})=S-\ddiv(\ub\phi).
\eeq
Multiplying the above equation by $A^{-1}(\phi-\overline{\phi})$, integrating by parts, we obtain that
\bea
&&\frac{1}{2}\frac{d}{dt}\|A^{-\frac{1}{2}}(\phi-\overline{\phi})\|^2+\|\nabla\phi\|^2+\int_{\Omega}(f'(\phi)-\overline{f'(\phi)})
(\phi-\overline\phi)dx\non\\
&=&\int_{\Omega}\left(S-\ddiv(\ub\phi)\right)A^{-1}(\phi-\overline{\phi})dx.\label{h-1es}
\eea
By Young's inequality, we have
\bea
&& \int_{\Omega}(f'(\phi)-\overline{f'(\phi)})
(\phi-\overline\phi)dx=\int_{\Omega}f'(\phi)(\phi-\overline\phi)dx\non\\
&=& \int_\Omega (\phi^3-\phi)(\phi-\overline{\phi}) dx\non\\
&=&\int_{\Omega}(\phi^4-\phi^2) dx -|\overline\phi|\int_\Omega \phi^3dx +|\Omega||\overline \phi|^2\non\\
&=&\int_{\omega}(2f(\phi)+\frac{1}{2}\phi^4)dx-|\overline\phi|\int_\Omega \phi^3dx +|\Omega||\overline \phi|^2\non\\
&\geq&2\int_{\Omega}f(\phi)dx-C_2.\eea
Moreover, by Young's inequality and Poincar\'{e}'s
 inequality, the right-hand side of  \eqref{h-1es} can be estimated as follows
\bea &&
\int_{\Omega}\left(S-\ddiv(\ub\phi)\right)A^{-1}(\phi-\overline{\phi})dx\non\\
&\leq&\int_{\Omega}S A^{-1}(\phi-\overline{\phi})dx+\int_{\Omega}\phi\ub \cdot\nabla A^{-1}(\phi-\overline{\phi})dx\non\\
&\leq&\|A^{-1}(\phi-\overline{\phi})\|\|S\|+\|\ub\|\|\phi\|_{L^4}\|\nabla A^{-1}(\phi-\overline{\phi})\|_{L^4}\non\\
&\leq&\frac{1}{2}\|\nabla\phi\|^2+C\|S\|^2+\frac{1}{2\eta}\|\ub\|^2+\frac{\eta}{2}\|\phi\|_{L^4}^2\|\nabla A^{-1}(\phi-\overline{\phi})\|_{L^4}^2\non\\
&\leq&\frac{1}{2}\|\nabla\phi\|^2+\frac{1}{2\eta}\|\ub\|^2+C\eta\|\phi\|_{L^4}^2(\|\phi\|_{L^4}^2+|\overline{\phi}|^2)+C\|S\|^2\non\\
&\leq&\frac{1}{2}\|\nabla\phi\|^2+\frac{1}{2\eta}\|\ub\|^2+(C_3\eta\|\phi\|_{L^4}^4+C_3M^2\eta\|\phi\|_{L^4}^2)+C_3\|S\|^2,\non
\eea
where $\eta>0$ is a constant to be specified later.  Since
\bea
C_3\eta\|\phi\|_{L^4}^4+C_3M^2\eta\|\phi\|_{L^4}^2
&\leq& C_3\eta\left(1+\frac{M^2}{4}\right)\|\phi\|^4_{L^4}+C_3 M^2\eta\non\\
&\leq&C_3\eta(8+2M^2)\int_{\Omega}f(\phi)dx+C_3(4+2M^2)\eta,\non
\eea
we take $\eta=\frac{1}{C_3(8+2M^2)}$ and deduce that
\beq \label{ab1}
\frac{d}{dt}\|A^{-\frac{1}{2}}(\phi-\overline{\phi})\|^2+\|\nabla\phi\|^2+2\int_{\Omega}f(\phi)dx
\leq C\|S\|^2+C_3(8+2M^2)\|\ub\|^2+C_4.
\eeq
%Next, we revisit \eqref{p00} by observing that in the 2D case,\beq\|\mu-\overline{\mu}\|_{L^s}\leq C\|\mu-\overline{\mu}\|_{H^1}\leq C\|\nabla\mu\|,\qquad\forall s>1.\eeq Thus we infer by H\"{o}lder's inequality  that
%\beq\begin{split}\label{p3}
%\|p\|\leq& C(\|S\|+\|\mu-\overline{\mu}\|_{L^{\beta' r}}\|\nabla\phi\|_{L^{\beta r}}+|\overline{\mu}|\|\phi-\overline{\phi}\|_{L^r})\\
%\leq&C(\|S\|+\|\nabla\mu\|\|\nabla\phi\|_{L^{\beta r}}+|\overline{\mu}|\|\phi-\overline{\phi}\|_{L^r})\end{split}\eeq
%where $\beta,\beta'>1$ such that \beq\frac{1}{\beta}+\frac{1}{\beta'}=1.\eeq Hence, \beq\|p\|_{L^{\frac{3}{2}}}\leq C(\|S\|+\|\nabla\mu\|\|\nabla\phi\|_{L^{s}}+|\overline{\mu}|\|\phi-\overline{\phi}\|_{L^{\frac{3}{2}}}),\qquad \forall s>\frac{3}{2}.\eeq
Multiplying \eqref{ab1} by $C_5=\frac{1}{C_3(16+4M^2)}$ and adding the resultant up with \eqref{est1a2} gives
\bea
 &&\frac{d}{dt}\left(E(\phi)+C_5\|A^{-\frac{1}{2}}(\phi-\overline{\phi})\|^2\right)\non\\
 && +\frac{1}{2}\|\nabla \mu\|^2+\frac{1}{2}\|\ub\|^2+C_5\|\nabla\phi\|^2+2C_5\int_{\Omega}f(\phi)dx\non\\
&\leq& C_6\|S\|^2\left(1+\|\phi\|_{L^4}^{\frac83}+\|\nabla\phi\|^\frac43\right)+C_7.
\label{ab2}
\eea
It is easy to see that there exist constants $C_8$, $C_9$ that are independent of $\phi$ such that
\beq
\nonumber C_8(\|\nabla\phi\|^2+\|\phi\|^4_{L^4})-C_9\leq E(\phi)+C_5\|A^{-\frac{1}{2}}(\phi-\overline{\phi})\|^2\leq C_8(\|\nabla\phi\|^2+\|\phi\|^4_{L^4})+C_9.
\eeq
Then we define $\Psi_1(t):=E(\phi)+C_5\|A^{-\frac{1}{2}}(\phi-\overline{\phi})\|^2+C_9+1$, which satisfies
\beq
\Psi_1(t)\geq \max\left\{1,C_8(\|\nabla\phi\|^2+\|\phi\|^4_{L^4})\right\}.\label{PHI1}
\eeq
Then it follows from \eqref{ab2}  and  Young's inequality that
\beq\label{ab3}
\frac{d}{dt}\Psi_1(t)+C_{10}\Psi_1(t)+\frac12\|\nabla\mu\|^2+\frac12\|\ub\|^2\leq C_{11}\|S\|^2\Psi_1^{\frac23}(t)+C_{11}(1+\|S\|^2).
\eeq
Since $S\in L^2_b(\mathbb{R}; \dot L^2(\Omega))$, then applying Lemma \ref{Gron1} in Appendix with $n=1$ and $\omega=a_1=\frac23$, we obtain the following dissipative estimates
\beq
\label{disphi}
\Psi_1(t)\leq C_{13}\Psi_1(\tau)e^{-\frac34 C_{10}(t-\tau)}+\mathcal{Q}_1\left(\|S\|_{L^2_b(\mathbb{R};\dot L^2(\Omega))}^2\right),\quad \forall\,t\geq\tau.
\eeq
It follows from the above estimate and \eqref{PHI1} that
\beq
\|\phi(t)\|_{H^1}^2\leq \mathcal{Q}_2(\|\phi_\tau\|_{H^1}^2)e^{-C_{14}(t-\tau)}+\mathcal{Q}_3\left(\|S\|_{L^2_b(\mathbb{R};\dot L^2(\Omega))}^2\right).\label{h1ab}
\eeq

As a consequence, we deduce from \eqref{h1ab} that for any $t\in\mathbb{R}$, $\hat D\in \mathcal{D}_F^{\mathcal{H}_M}$, there exists a time $\tau_1(\hat D, t)<t-3$ such that
\beq
\|\phi(r; \tau, \phi_\tau)\|_{H^1}^2\leq \rho_1, \quad \forall\, r\in[t-3,t],\ \tau\leq \tau_1(\hat D, t), \ \phi_\tau\in D\in \hat D,
\eeq
where $$\rho_1=1+\mathcal{Q}_3\left(\|S\|_{L^2_b(\mathbb{R};\dot L^2(\Omega))}^2\right).$$
Besides, integrating \eqref{ab3}, we infer that
\beq
\sup_{r\in[t-2,t]}\int_{r-1}^{r}\left(\|\nabla\mu(s)\|^2+\|\ub(s)\|^2\right)ds
\leq \mathcal{Q}_4\left(\rho_1, \|S\|_{L^2_b(\mathbb{R};\dot L^2(\Omega))}^2\right).\label{bb0}
\eeq
for $\tau\leq \tau_1(\hat D, t)$ and $\phi_\tau\in D\in \hat D$, which together with \eqref{h1ab} and the Sobolev embedding theorem yields
\beq
\sup_{r\in[t-2,t]}\int_{r-1}^{r}\|\phi\|_{H^3}^2 ds
\leq \mathcal{Q}_5\left(\rho_1, \|S\|_{L^2_b(\mathbb{R};\dot L^2(\Omega))}^2\right).\label{bb00}
\eeq
Next, testing \eqref{3} by $\Delta^2\phi$, using the estimate \eqref{h1ab} and a similar argument in Lemma \ref{regww}, we can still obtain the differential inequality \eqref{ap4} for $\|\Delta\phi\|^2$, namely,
\beq\begin{split}\label{ap4a}
\frac{d}{ds}\|\Delta\phi(s)\|^2+\|\Delta^2\phi(s)\|^2\leq Ch(s)\|\Delta\phi\|^2+Ch(s),
\end{split}
\eeq
for a.e. $s\in [t-3,t]$, $\tau\leq \tau_1(\hat D, t)$ and $\phi_\tau\in D\in \hat D$, here $h(s)=1+\|S\|^2+\|\nabla\Delta\phi\|^2$, and the constant $C$ now depends on $\rho_1$, $\Omega$ and $\|S\|_{L^{2}_b(\mathbb{R};L^{2})}$.

Using \eqref{bb00}, \eqref{ap4a} and the uniform Gronwall inequality \cite[Lemma III.1.1]{TE}, we can deduce that
\beq\label{delab}
 \|\Delta\phi(r)\|^2\leq\mathcal{Q}_6\left(\rho_1, \|S\|_{L^2_b(\mathbb{R};\dot L^2(\Omega))}^2\right),\quad\forall\, r\in[t-2,t].
 \eeq
Thus, it follows from \eqref{h1ab} and \eqref{delab} that
\beq\label{h2ab}
\|\phi(r; \tau, \phi_\tau)\|_{H^2}^2\leq \rho_2, \quad \forall\, r\in[t-2,t],\ \tau\leq \tau_1(\hat D, t), \ \phi_\tau\in D\in \hat D
\eeq
where $\rho_2$ depends on $\rho_1$, $\|S\|_{L^2_b(\mathbb{R};\dot L^2(\Omega))}^2$, $M$ and $\Omega$.

In summary, we can take the family
$$\hat D_0=\left\{D_0(t)=\mathcal{B}_M(0,\rho_2^\frac12), \ t\in \mathbb{R}\right\}\in \mathcal{D}^{\mathcal{H}_M}_F,$$
where $\mathcal{B}_M(0,\rho_2^\frac12)$ is the closed ball in $\mathcal{H}_M$ of center zero and radius $\rho_2^\frac12$. Then $\hat D_0$ satisfies that for any $t\in \mathbb{R}$ and any family $\hat D\in \mathcal{D}^{\mathcal{H}_M}_F$, there exists a time $\tau_0(\hat D, t)<t$ such that
$$U(t,\tau)D(\tau)\subset D_0(t), \quad \forall\, \tau\leq \tau_0(\hat D, t), \ D(t)\in \hat D.$$
 This completes the proof.
\end{proof}

Using the uniform estimates obtained in the above proposition and the Sobolev embedding theorem, indeed we can also prove the following
\begin{corollary}\label{hies1ab}
For any $t\in \mathbb{R}$ and any family $\hat D\in \mathcal{D}^{\mathcal{H}_M}_F$, there exists a time $\tau_0(\hat D, t)<t$ such that
\beq\non
\sup_{r\in[t-1,t]}\int_{r-1}^r\left(\|\phi(s)\|_{H^4}^2+\|\ub(s)\|_{H^1}^2+\|\phi_t(s)\|^2\right)ds\leq \rho_3,\quad \forall\, \tau\leq \tau_0(\hat D, t),\ \phi_\tau\in D(\tau).
\eeq
\end{corollary}

\subsection{Pullback $\mathcal{D}_F^{\mathcal{H}_M}$-asymptotic compactness}

Now we proceed to prove the pullbak $\mathcal{D}_F^{\mathcal{H}_M}$-asymptotic compactness  for the universe $\mathcal{D}_F^{\mathcal{H}_M}$ in $\mathcal{H}_M$.

\begin{proposition}\label{compact} Suppose that $S\in L^2_b(\mathbb{R}; \dot L^2(\Omega))$. Then the family of process $\{U(t,\tau)\}_{t\geq \tau}$ is pullback $\mathcal{D}^{\mathcal{H}_M}_F$-asymptotically compact.
\end{proposition}
\begin{proof}
Consider $t\in\mathbb{R}$, a family $\hat{D}\in\mathcal{D}_F^{\mathcal{H}_M}$, a sequence of time $\tau_n\to -\infty$ and
a sequence of initial data $\phi_{\tau_n}\in D(\tau_n)\in \hat{D}$ (recall from the definition that here the set $D(t)$ is indeed time independent). For the sake of simplicity, below we just denote $$\phi^n(s)=\phi(s; \tau_n, \phi_{\tau_n})=U(s, \tau_n)\phi_{\tau_n}.$$

It follows from Proposition \ref{plab} and Corollary \ref{hies1ab} that there exists  a $\tau_0(\hat D, t)<t-3$ such that the subsequence $\{\phi^n:\ \tau_n\leq \tau_0(\hat D, t)\}\subset\{\phi^n\}$ is uniformly bounded in $L^\infty(t-2, t; H^2(\Omega)\cap L^2(t-2, t; H^4(\Omega))$ and correspondingly, $\{\phi_t^n\}$ is uniformly bounded in $L^2(t-2, t; L^2(\Omega))$.

Recall the following compactness lemma (see e.g., \cite{SI}),
\begin{lemma}
Let $X\subset Y\subset Z$ be three Hilbert spaces, $T\in(0,+\infty)$. Suppose that the embedding $X\hookrightarrow Y$ is compact. Then
\par
(1) For any $p, q\in (1, +\infty)$, the embedding $\{\phi\in L^{p}(0, T; X), \ \phi_t\in L^{q}(0, T; Z)\}\hookrightarrow L^{p}(0, T; Y)$ is compact.
\par (2) For any $q\in (1, +\infty)$, the embedding $\{\phi\in L^{\infty}(0, T; X), \ \phi_t\in L^{q}(0, T; Z)\}\hookrightarrow C([0, T]; Y)$ is compact.
\par
(3) The embedding $\{\phi\in L^{2}(0, T; X), \ \phi_t\in L^{2}(0, T; Y)\}\hookrightarrow C([0, T]; [X,Y]_{\frac12})$ is continuous.
\end{lemma}
We deduce that there exists a subsequence still denoted by $\{\phi^n\}$ and a function $\phi\in L^\infty([t-2, t]; H^2(\Omega)\cap L^2(t-2, t; H^4(\Omega))$ with $\phi_t\in L^2(t-2, t; L^2(\Omega))$ such that
\bea
&&\phi^n\rightharpoonup \phi, \quad \text{weakly star in} \quad L^\infty(t-2, t; H^2(\Omega)),\non\\
&&\phi^n\rightharpoonup \phi, \quad \text{weakly in} \quad L^2(t-2, t; H^4(\Omega)),\non\\
&&\phi^n_t\rightharpoonup \phi_t, \!\!\quad \text{weakly in}\quad  L^2(t-2, t; L^2(\Omega)),\non\\
&&\phi^n\to \phi, \quad \text{strongly in}\quad  L^2(t-2, t; H^2(\Omega))\ \text{and}\ C([t-2, t], H^1(\Omega)),\label{strongh1}\\
&&\phi^n(s)\to \phi(s), \quad \text{strongly in}\quad  H^2(\Omega), \ \text{for a.e.} \ s\in(t-2, t).\label{strongae}
\eea
Moreover, we have $\phi\in C([t-2, t], H^2(\Omega))$ and it satisfies the system \eqref{1}--\eqref{3} a.e. on $(t-2, t)$.

  From the fact that $\{\phi^n\}$ is uniformly bounded in $C([t-2, t], H^2(\Omega))$, we infer that for any sequence $\{s_n\}\subset[t-2,t]$ satisfying $s_n\to s_*\in [t-2,t]$, it holds (up to a subsequence)
\beq
\phi^n(s_n)\rightharpoonup \phi(s_*)\quad \text{weakly in} \ H^2(\Omega).\label{weakcon}
\eeq

In what follows, we prove that the sequence $\{\phi^n(t)\}$ is relatively compact in $\mathcal{H}$ (see Definition \ref{comp}), which is a direct consequence of the following result such that up to a subsequence, it holds
\beq
\phi^n\to \phi\quad \text{strongly in} \ C([t-1,t]; H^2(\Omega)).\label{claim}
\eeq

  To proceed, first we need to derive proper energy estimates. For every $\phi^n$, recalling \eqref{ap3} and the computations in \eqref{i1}--\eqref{i31ba}, using the interpolation inequality $\|\nabla \Delta\phi^n\|^2\leq \|\Delta\phi^n\|\|\Delta^2\phi^n\|$ and Young's inequality, after a straightforward but tedious calculation, we can re-estimate the three terms on the right-hand side of \eqref{ap3} (now in terms of $\phi^n$, cf. \eqref{i1}--\eqref{i31ba}) and deduce that
\beq\begin{split}\label{ap4c}
\frac{d}{dt}\|\Delta\phi^n\|^2+\|\Delta^2\phi^n\|^2\leq C_\Omega(F_1(\phi^n)+F_2(\phi^n)+F_3(\phi^n)),
\end{split}
\eeq
where $C_\Omega$ is a constant that  depends only on $\Omega$. In particular, it is independent of $\phi^n$. The functions $F_i$ are given by
\bea
F_1(\phi^n)& = & \|\phi^n\|_{H^1}^4\|\Delta\phi^n\|^6,\non\\
F_2(\phi^n)&=& (\|\phi^n\|_{H^1}^{16}+\|S\|^2+1)\|\Delta\phi^n\|^2,\non\\
F_3(\phi^n)&=& \|\phi^n\|_{H^1}^{10}+\|S\|^2\|\phi^n\|_{H^1}^2+1.\non
\eea
In a similar manner, we have for $\phi$
\beq\begin{split}\label{ap4d}
\frac{d}{dt}\|\Delta\phi\|^2+\|\Delta^2\phi\|^2\leq C_\Omega(F_1(\phi)+F_2(\phi)+F_3(\phi)),
\end{split}
\eeq
where $C_\Omega$ is the same as in \eqref{ap4c}.

As a consequence, for $\phi^n$ and $\phi$, $t-2\leq s_1\leq s_2\leq t$, we infer from the above inequalities that
\bea
&& \|\Delta\phi^n(s_2)\|^2+\int_{s_1}^{s_2}\|\Delta^2\phi^n(\xi)\|^2d\xi\non\\
&\leq& \|\Delta\phi^n(s_1)\|^2+C_\Omega\int_{s_1}^{s_2}(F_1(\phi^n(\xi))+F_2(\phi^n(\xi))+F_3(\phi^n(\xi))) d\xi,\label{ap4e}\\
&& \|\Delta\phi(s_2)\|^2+\int_{s_1}^{s_2}\|\Delta^2\phi(\xi)\|^2d\xi\non\\
&\leq& \|\Delta\phi(s_1)\|^2+C_\Omega\int_{s_1}^{s_2}(F_1(\phi(\xi))+F_2(\phi(\xi))+F_3(\phi(\xi))) d\xi.
\eea
Define
\bea
J_n(s)&=&\|\Delta\phi^n(s)\|^2-C_\Omega\int_{t-2}^s(F_1(\phi^n(\xi))+F_2(\phi^n(\xi))+F_3(\phi^n(\xi))) d\xi,\non\\
J(s)&=&\|\Delta\phi(s)\|^2-C_\Omega\int_{t-2}^s(F_1(\phi(\xi))+F_2(\phi(\xi))+F_3(\phi(\xi))) d\xi.\non
\eea
Since $\phi^n, \phi\in C([t-2,t]; H^2(\Omega))$, the functions $J_n(s)$ and $J(s)$ are continuous for $s\in [t-2, t]$. Moreover, they are non-increasing with respect to $s\in [t-2, t]$. To this end,  we infer from \eqref{ap4e} that
\bea
&&J_n(s_2)-J_n(s_1)\non\\
&=& \|\Delta\phi^n(s_2)\|^2-\|\Delta\phi^n(s_1)\|^2-C_\Omega\int_{s_1}^{s_2}(F_1(\phi^n(\xi))+F_2(\phi^n(\xi))+F_3(\phi^n(\xi))) d\xi\non\\
&\leq& -\int_{s_1}^{s_2}\|\Delta^2\phi^n(\xi)\|^2d\xi\non\\
&\leq&0, \quad \text{for all}\ t-2\leq s_1\leq s_2\leq t.\non
\eea
Similar result holds for $J(s)$. From the strong convergence results \eqref{strongh1} and \eqref{strongae}, we have for a.e. $s\in (t-2, t)$, $\|\Delta \phi^n(s)\|\to \|\Delta \phi(s)\|$ and $\|\phi^n(s)\|_{H^1}\to \|\phi(s)\|_{H^1}$. As a consequence,
\beq
F_i(\phi^n(s))\to F(\phi(s)),\quad \text{a.e. for}\ s\in (t-2, t), \quad i=1,2,3.
\eeq
Since $\phi^n$ is uniformly bounded in $L^\infty(t-2, t; H^2(\Omega)$, then $F_i(\phi^n)$ is also bounded $L^\infty(t-2, t)$. It follows from the Lebesgue dominated convergence theorem that
\beq
\int_{t-2}^s F_i(\phi^n(\xi))d\xi\to \int_{t-2}^s F_i(\phi(\xi))d\xi, \quad \forall\,s\in [t-2,t], \ i=1,2,3,\label{conF}
\eeq
which implies
\beq
J_n(s)\to J(s), \quad \text{a.e.} \ s\in(t-2, t).\label{aaecon}
\eeq

Now we proceed to prove the strong convergence property \eqref{claim} by a contradiction argument introduced in \cite{GMR12, MPR09}. Assume that \eqref{claim} is not true, then there exists a constant $\kappa>0$ and a sequence $\{t_n\}_{n=1}^\infty\subset [t-1,t]$ that without loss of generality, converges to a certain point $t^*\in [t-1,t]$ (otherwise, we can take a convergent subsequence) such that
\beq \|\phi^n(t_n)-\phi(t^*)\|_{H^2}\geq 2\kappa.\non\eeq
From the elliptic estimate, here we can simply use the equivalent norm on $H^2(\Omega)$ given by $\|\cdot\|_{H^2}=\|\cdot\|_{H^1}+\|\Delta\cdot\|$. Then it follows from \eqref{strongh1} that there exists $n_0\in \mathbb{N}$ depending on $\kappa$ such that
\beq
\|\Delta \phi^n(t_n)-\Delta \phi(t^*)\|\geq \kappa,\quad \forall\, n\geq n_0.\label{contr1}
\eeq
On the other hand, from \eqref{aaecon}, we can take a monotone increasing sequence $\{r_j\}\subset(t-2, t^*)$ that satisfies
\beq
\lim_{j\to +\infty} r_j=t^*\quad \text{and}\ \lim_{n\to+\infty} J_n(r_j)=J(r_j), \quad \forall\, j\in \mathbb{N}.\label{JJ}
\eeq
For any $\delta>0$, it follows from the continuity of $J(s)$ that there exists a constant $j_0\in\mathbb{N}$ depending on $\delta$ such that \beq
|J(r_j)-J(t^*)|<\frac{\delta}{2},\quad \forall\,j\geq j_0(\delta).\label{JJ1}
\eeq
Due to \eqref{JJ}, for $j_0$, there exists an integer $n_1$ depending on $j_0$ and satisfying $n_1\geq n_0$ such that
\beq
t_n\geq r_{j_0}, \quad \text{and}\quad |J_n(r_{j_0})-J(r_{j_0})|<\frac{\delta}{2}, \quad \forall\, n\geq n_1.\label{JJ2}
\eeq
Since $J_n(s)$ is non-increasing for $s\in [t-2,t]$, we infer from \eqref{JJ1} and \eqref{JJ2} that for all $n\geq n_1$, it holds
\beq
J_n(t_n)-J(t^*)\leq J_n(r_{j_0})-J(t^*)\leq |J_n(r_{j_0})-J(r_{j_0})|+|J(r_{j_0})-J(t^*)|<\delta,
\eeq
which implies
\beq
\limsup_{n\to+\infty}J_n(t_n)\leq J(t^*).\label{supJ}
\eeq
It follows from \eqref{conF} and the boundedness of $F_i$ that
\bea
&& \lim_{n\to+\infty}\left|\int_{t-2}^{t_n} F_i(\phi^n(\xi))d\xi-\int_{t-2}^{t^*} F_i(\phi(\xi))d\xi\right|\non\\
&\leq&\lim_{n\to+\infty}\left|\int_{t-2}^{t^*} F_i(\phi^n(\xi))d\xi-\int_{t-2}^{t^*} F_i(\phi(\xi))d\xi\right|+\lim_{n\to+\infty}\left|\int_{t^*}^{t_n} F_i(\phi^n(\xi))d\xi\right|\non\\
&=&0, \quad i=1,2,3.\label{limFi}
\eea
Then from the definition of $J_n$, $J$, and \eqref{supJ}--\eqref{limFi}, we can see that
\beq
\limsup_{n\to+\infty}\|\Delta \phi^n(t_n)\|\leq \|\Delta \phi(t^*)\|.\label{supphi}
\eeq
On the other hand, the weak convergence \eqref{weakcon} implies that
\beq
\liminf_{n\to+\infty}\|\Delta \phi^n(t_n)\|\geq \|\Delta \phi(t^*)\|.\label{supphi1}
\eeq
As a consequence, we have the norm convergence
\beq
\lim_{n\to+\infty}\|\Delta \phi^n(t_n)\|= \|\Delta \phi(t^*)\|,\label{supphi2}
\eeq
which together with the weak convergence \eqref{weakcon} yields the strong convergence such that
\beq
\lim_{n\to+\infty}\|\Delta \phi^n(t_n)-\Delta \phi(t^*)\|=0.\label{supphi3}
\eeq
This leads to a contradiction with our assumption \eqref{contr1}. Therefore, \eqref{claim} holds and the sequence $\{\phi^n(t)\}$ is relatively compact in $\mathcal{H}$. The proof is complete.
\end{proof}

\subsection{Proof of Theorem \ref{T2}}

For any $S\in L^2_b(\mathbb{R}; \dot L^2(\Omega))$, we know from Proposition \ref{process} that the global strong solution $\phi$ to problem \eqref{1}--\eqref{ini} defines a closed process $\{U(t,\tau)\}_{t\geq\tau}$ in the phase space $\mathcal{H}_M$. Observing Propositions \ref{plab} and \ref{compact}, also noticing that the pullback $\mathcal{D}_F^{\mathcal{H}_M}$-absorbing family $\hat D_0$ constructed in Proposition \ref{plab} indeed belongs to the universe $\mathcal{D}_F^{\mathcal{H}_M}$, then we are able to apply the abstract results in Lemma \ref{34} and Remark \ref{pullre} to conclude that the process $\{U(t,\tau)\}_{t\geq\tau}$ admits a minimal pullback $\mathcal{D}_F^{\mathcal{H}_M}$-attractor
$\mathcal{A}_{\mathcal{D}_F^{\mathcal{H}_M}}=\{A_{\mathcal{D}_F^{\mathcal{H}_M}}(t): t\in\mathbb{R}\}$ in $\mathcal{H}_M$, which is given by
$$
A_{\mathcal{D}_F^{\mathcal{H}_M}}(t)=\Lambda(\hat D_0, t)= \bigcap_{s\leq t}\overline{\bigcup_{\tau\leq s}U(t,\tau)D_0(\tau)}^{H^2(\Omega)}.
$$
The proof of Theorem \ref{T2} is complete.

\begin{remark} We remark that in the current particular case under consideration, i.e., $\hat D$ is parameterized in time but constant for all $t\in \mathbb{R}$, the corresponding minimal pullback $\mathcal{D}_F^{\mathcal{H}_M}$-attractor for the process $\{U(t, \tau)\}_{t\geq \tau}$ is just the pullback attractor defined in \cite{CDF97}. One can also apply the abstract results in \cite{GMR12} to treat more general case that the family $\hat D$ is time dependent, under suitable assumptions on its element $D$ and the external source term $S$ . We leave this to the interested reader.
\end{remark}

\section{Convergence to Steady States in 2D}

In this section, we investigate the long-time behavior of a single trajectory $\phi(t)$ when the associated dynamical process becomes \emph{asymptotically autonomous} as time goes to infinity.

\subsection{Uniform-in-time estimates}
Hereafter, we assume that the external source term $S$ satisfies
\beq
S\in L^2(\tau,+\infty;\dot L^2(\Omega)).\label{CS1}
\eeq
We recall the inequality \eqref{est1a} which implies that
\beq
\frac{d}{dt} E_0(\phi_n)+\frac12\|\nabla \mu_n\|^2+\|\ub_n\|^2\leq C\|S\|^2E_0(\phi_n),
\eeq
\beq
 E_0(\phi_n(t))\leq E_0(\phi_\tau)e^{\int_\tau^t\|S\|^2ds},\quad \forall\, t\geq \tau,\non
\eeq
 The above estimate easily yields the following uniform-in-time estimates for  global weak (or strong) solutions to problem \eqref{1}--\eqref{ini} such that
\beq
\sup_{t\in[\tau, +\infty)}\|\phi(t)\|_{H^1}^2+\int_{\tau}^{+\infty}\|\nabla
\mu\|^2dt+\int_{\tau}^{+\infty}\|\ub\|^2dt
\leq C,\label{unih1a}
\eeq
and
\beq
\sup_{t\geq \tau}\int_{t}^{t+1}\|\phi\|_{H^3}^2ds\leq C,\label{unih3}
\eeq
where the constant $C$  depends only on $\|\phi_\tau\|_{H^1}$, $\int_\tau^{+\infty}\|S\|^2ds$ and $\Omega$.

Next, recalling the differential inequality \eqref{ap4}, by the uniform Gronwall inequality \cite[Lemma III.1.1]{TE}, we can deduce that
\beq \label{est2q}
\|\Delta\phi(t+1)\|^2\leq C,\quad \forall\ t\geq \tau,
\eeq
where the constant $C$ depends on $\|\phi_\tau\|_{H^1}$, $\Omega$ and $\int_\tau^{+\infty}\|S\|^2ds$. If in addition, $\phi_\tau \in H^2(\Omega)$, then by the classical Gronwall inequality, we have
\beq
\|\Delta\phi(t)\|^2 \leq (\|\Delta\phi_\tau\|^2+1)e^{C\int_\tau^{\tau+1} h(s)ds}\leq C,\quad \forall\, t\in[\tau, \tau+1]. \label{est2qa}
\eeq
The above uniform-in-time estimates  \eqref{est2q}--\eqref{est2qa} imply that

\begin{proposition}\label{regg}
Assume that $S\in L^2(\tau,+\infty;\dot L^2(\Omega))$. Then the global strong solution to problem \eqref{1}--\eqref{ini} is uniformly bounded in $H^2$ for all $t\geq \tau$. Moreover, the global weak solution to problem \eqref{1}--\eqref{ini} will become a strong one after a positive time and it is also uniformly bounded in $H^2$.
\end{proposition}

\subsection{The $\omega$-limit set}

Since we are interested in the long-time behavior of $\phi$ as $t\to+\infty$,  Proposition \ref{regg} enables us to focus on the study of uniformly bounded global strong solution of problem \eqref{1}--\eqref{ini}.

For any initial datum $\phi_\tau\in H^2_N(\Omega)$. We define the $\omega$-limit set as follows
\beq\nonumber
\omega(\phi_\tau)=\{\phi_{\infty}\in H^2_N(\Omega)\ | \ \exists \{t_n\}\nearrow+\infty\ \text{s.t.}\ \phi(t_n)\rightarrow \phi_{\infty}\ \text{in}\ H^1,\ \text{as}\ t_n\rightarrow+\infty\}.
\eeq
Besides, we introduce the set of steady states associated with the initial datum
\beq\mathcal{S}=\left\{\psi\in H^2_N(\Omega)\; |-\Delta\psi+f'(\psi)=\frac{1}{|\Omega|}\int_{\Omega}f'(\psi) dx, \text{ a.e. in } \Omega, \int_{\Omega}\psi dx=\int_{\Omega}\phi_\tau dx\right\}.\label{SSS}
\eeq
Using the classical variational method and the elliptic regularity theorem, we can easily deduce that (see \cite[Proposition 3.5]{WW2012} for the case with periodic boundary condition)

\begin{proposition}\label{esci}
 The set $\mathcal{S}$ is nonempty. Any element $\psi\in \mathcal{S}$ is a critical point of $E(\phi)$, which satisfies
 $\psi\in C^\infty$ and its $H^m$-norms ($m\geq 0$) are bounded by a constant depending on $|\overline{
\phi_\tau}|$ and $\Omega$.
 \end{proposition}

Using the fact that the strong solution $\phi$ is uniformly bounded in $H^2$ for $t\geq \tau$, similar to the calculations in \eqref{app0}--\eqref{est1a0} for the approximate solution, we can apply Young's inequality to obtain the following energy inequality for $\phi$:
\beq
\frac{d}{dt}E(\phi(t))+\frac{1}{2}\|\nabla\mu\|^2+\|\ub\|^2\leq K_1\|S\|^2,\quad\text{for a.e.}\ t\geq \tau,
\label{BEL}
\eeq
where
\beq
 E(\phi)=\int_\Omega\left( \frac{1}{2}|\nabla\phi|^2+f(\phi)\right)dx\label{E}
  \eeq
and $K_1$ is a constant depending on $\|\phi_\tau\|_{H^2}$, $\int_\tau^{+\infty}\|S\|^2 ds$ and $\Omega$.

The above type of energy inequality plays an important role in studying the long-time behavior of global solutions to non-autonomous system (cf. \cite{HT01,CJ}). First, we can prove the following relationship between the $\omega$-limit set and set $\mathcal{S}$.

\begin{proposition}\label{omega}
For any $\phi_\tau\in H^2_N(\Omega)$, its corresponding $\omega$-limit set is a nonempty bounded subset in $H^2(\Omega)$ such that $\omega(\phi_\tau)\subset \mathcal{S}$. Moreover, $E(\phi)$ is a constant on $\omega(\phi_\tau)$.
\end{proposition}
\begin{proof}
Due to the uniform $H^2$-estimate for $\phi$ and the compact embedding $H^2\hookrightarrow H^1$, there exists certain function  $\phi_\infty\in H^2_N(\Omega)$ and a unbounded increasing sequence $t_n\to+\infty$ that $\|\phi(t_n)-\phi_\infty\|_{H^1}\to 0$ as $n\to+\infty$. Hence, $\omega(\phi_\tau)$ is a nonempty, bounded subset in $H^2(\Omega)$.

It follows from \eqref{BEL} that
\beq
E(\phi(t_1))-E(\phi(t_2))\leq K_1\int_{t_2}^{t_1}\|S\|^2dt, \quad \forall\, \tau\leq t_2\leq t_1<+\infty.\label{contiE}
\eeq
Thus, $E(\phi(t))$ is continuous in time (and it is bounded from below from its definition \eqref{E}).

Denote $\tilde{E}(t)=E(\phi(t))+K_1\int_t^{\infty}\|S\|^2ds$. Then it follow from \eqref{BEL} that
\beq\nonumber \frac{d}{dt}\tilde{E}(t)+\frac12\|\nabla\mu\|^2+\|\ub\|^2\leq 0,\quad\text{for  }t\geq \tau.
\eeq Hence, $\tilde{E}(t)$ is non-increasing in $t$. Since $\tilde{E}$ is also bounded from below, we may infer that as $t\rightarrow+\infty$, $\tilde{E}(t)\rightarrow E_{\infty}$ for some constant $E_{\infty}$.  Recalling the fact $\lim_{t\to+\infty}\int_t^{+\infty}\|S\|^2 ds=0$, we get
\beq
\lim_{t\to +\infty} E(\phi(t))=E_\infty.\label{limE}
\eeq
By the definition of $\omega(\phi_\tau)$, it is easy to see that $E(t)$ equals $E_{\infty}$ on $\omega(\phi_\tau)$.

Next, for any cluster point $\phi_\infty \in \omega(\phi_\tau)$, it easily follows that $\phi_\infty\in H^2_N(\Omega)$ and $\overline{\phi_\infty}=\overline{\phi_\tau}$. In order to show that $\phi_\infty\in \mathcal{S}$, we apply the argument introduced in \cite{HT01}. Consider the unbounded increasing sequence $t_n\to+\infty$ such that $\|\phi(t_n)-\phi_\infty\|_{H^1}\to 0$ as $n\to+\infty$. Without loss of generality, we assume $t_{n+1}\geq
t_n+1,\ n\in\mathbb{N}$.  Integrating
\eqref{BEL} on the time interval $[t_n,t_{n+1}]$, we obtain that
 \bea && E(\phi(t_{n+1}))- E(\phi(t_n)+ \int^{t_{n+1}}_{t_n}
 \left(\frac12 \|\nabla \mu(s)\|^2+\|\ub(s)\|^2\right)ds
 \non\\
 &\leq& K_1 \int^{t_{n+1}}_{t_n}
 \|S\|^2ds.\label{Ly5}
 \eea
It follows from \eqref{limE} and \eqref{Ly5} that as $n\rightarrow +\infty$, it holds
 \bea & & \int^{1}_{0}
 \left(\frac12 \|\nabla \mu(t_n+s)\|^2+\|\ub(t_n+s)\|^2\right)ds\non\\
 &\leq & \int^{t_{n+1}}_{t_n}
 \left(\frac12 \|\nabla \mu(s)\|^2+\|\ub(s)\|^2\right)ds \rightarrow 0.
 \label{Ly6}
 \eea
Besides, by equation \eqref{1}, the uniform $H^2$-estimate for $\phi$ and Agmon's inequality, we have (cf. \cite{Ab})
 \bea \|\phi_t\|_{(H^1(\Omega))'}
 &\leq& C(\|\ub\phi\| + \|\nabla \mu\|+\|S\|)\leq
  C(\|{\bf u}\|\|\phi\|_{L^\infty}+\|\nabla\mu\|+\|S\|)\non\\
&\leq& K_2\left(\|\ub\|+ \|\nabla \mu\|+\|S\|\right),\label{dt}
 \eea
where $K_2$ is a constant depending on $\|\phi_\tau\|_{H^2}$, $\int_\tau^{+\infty}\|S\|^2 ds$ and $\Omega$.
By \eqref{dt} and \eqref{Ly6}, we have
 \beq
 \lim_{n\to+\infty}\int^{1}_{0}\|\phi_t(t_n+s)\|_{(H^1(\Omega))'}^2ds=0.
 \eeq
As a consequence,
 \beq \|\phi(t_n+s_1)-\phi(t_n+s_2)\|_{(H^1(\Omega))'}\rightarrow 0,\quad \text{uniformly} \
 \text{for all}\ s_1, s_2\in [0,1].\non
 \eeq
From the precompactness of $\phi(t)$ in $H^1(\Omega)$ and the sequential convergence of $\phi(t_n)$ in $H^1$, we infer that
 \beq \lim_{n\rightarrow
 \infty}\|\phi(t_n+s)-\phi_\infty\|_{H^1}=0,\quad \forall \,
 s\in[0,1].\label{ffff}
 \eeq
For any $\xi\in H^1(\Omega)$, using Lebesgue dominated convergence
theorem, the Poincar\'e inequality, \eqref{Ly6} and \eqref{ffff}, we deduce that
\bea
&&\left|\int_\Omega (\nabla \phi_\infty \cdot \nabla \xi+f'(\phi_\infty)\xi-\overline{f'(\phi_\infty)}\xi)dx\right|\non\\
&=& \lim_{n\to+\infty}\left|\int^{1}_{0} \int_\Omega\left(\nabla \phi(t_n+s)\cdot \nabla \xi+f'(\phi(t_n+s))\xi-\overline{f'(\phi(t_n+s))}\xi\right) dx ds\right|\non\\
&=& \lim_{n\to+\infty}\left|\int^{1}_{0}\int_\Omega (\mu(t_n+s)-\overline{\mu}(t_n+s))\xi dxds \right|\non\\
&\leq & \lim_{n\to+\infty}\int^{1}_{0}
  \|\mu(t_n+s)-\overline{\mu}(t_n+s)\|\|\xi\| ds\non\\
&\leq&
\lim_{n\to+\infty}\left(\int^{1}_{0}
  \|\mu(t_n+s)-\overline{\mu}(t_n+s)\|^2 ds\right)^\frac12 \|\xi\|\non\\
  &\leq& \lim_{n\to+\infty}C\left(\int^{1}_{0}
  \|\nabla \mu(t_n+s)\|^2 ds \right)^\frac12 \|\xi\|\non\\
  &=&0\non
\eea
which enables us to conclude that $\phi_\infty\in \mathcal{S}$. The proof is complete.
\end{proof}
\begin{remark}
Indeed, from \eqref{Ly5}, we can also obtain the decay of velocity $\ub$ in the following weak sense
$$\lim_{t\to+\infty}\int^{1}_{0}  \|\ub(t+s)\|^2ds=0. $$

\end{remark}

%In this part, we show the decay property of the energy dissipation. We shall make use of the following lemma given by Shen \& Zheng \cite{ShZ}.
%\begin{lemma}\label{lmSZ} Suppose $y(t)$ and $h(t)$ are nonnegative functions, $y'(t)$ is locally integrable on $(0,+\infty)$ and $y(t)$, $h(t)$ satisfy
%\beq\frac{dy}{dt}\leq A_1y^2+A_2+h(t),\qquad \forall t\geq0,\eeq
%\beq\int_0^Ty(\tau)d\tau\leq A_3,\qquad\int_0^Th(\tau)d\tau\leq A_4,\quad \forall T>0\eeq with $A_1$, $A_2$, $A_3$, $A_4$ being positive constants independent of $t$, $T$. Then for any $r>0$, \beq y(t+r)\leq\left(\frac{A_3}{r}+A_2r+A_4\right)e^{A_1A_3},\qquad \forall t\geq0.\eeq
%Moreover, if $T=+\infty$, then\beq \lim_{t\rightarrow+\infty}y(t)=0.\eeq
%\end{lemma}

\subsection{Convergence of trajectory $\phi(t)$}
The precompactness of the trajectory $\phi(t)$ in $H^1(\Omega)$ only yields a sequential convergence result for $\phi(t)$.
Next, we demonstrate that the $\omega$-limit set $\omega(\phi_\tau)$ consists of a single point, namely, we show that each bounded global strong solution converges to a single steady state as time goes to infinity. For this purpose, we assume in addition that
\beq\label{AS}
\sup\limits_{t\geq\tau}(1+t)^{1+\rho}\int_t^{+\infty}\|S\|^2 ds<+\infty,\quad\text{for some } \rho>0.
\eeq
 First, we introduce the following \L ojasiewicz-Simon type inequality, which easily follows from the abstract result in \cite{GG06}:
\begin{lemma} \label{LS} Let $\psi\in H^2_N(\Omega)$ be a critical point of $E(\phi)$. Then there exist constants $\theta\in(0,\frac{1}{2})$ and $\beta>0$ depending on $\psi$ such that for any $\phi\in H^2_N(\Omega)$ satisfying $\int_{\Omega}\phi dx=\int_{\Omega}\psi dx$ and $\|\phi-\psi\|_{H^1}\leq \beta$, it holds that
\beq\label{ls}
\|\mathrm{P}_0(-\Delta\phi+f'(\phi))\|\geq |E(\phi)-E(\psi)|^{1-\theta}.
\eeq
\end{lemma}

 The proof for convergence of the whole trajectory $\phi(t)$ follows from the so-called \L ojasiewicz-Simon approach (see e.g., \cite{CJ, FS, HT01, GG10, ZWH}). By Lemma \ref{LS}, for each element $\phi_\infty\in\omega(\phi_\tau)$, there exists a $\beta_{\phi_\infty}>0$ and $\theta_{\phi_\infty}\in(0,\frac{1}{2})$ such that
the inequality \eqref{ls} holds for
\beq \phi\in \textbf{B}_{\beta_{\phi_\infty}}(\phi_\infty):=\Big\{\phi\in H^2_N(\Omega):\int_{\Omega}\phi dx=\int_{\Omega}\phi_\tau dx, \  \ \|\phi-\phi_\infty\|_{H^1}<\beta_{\phi_\infty}\Big\}.\non
\eeq
The union of  balls $\{\textbf{B}_{\beta_{\phi_\infty}}(\phi_\infty):\phi_\infty\in\omega(\phi_\tau)\}$ forms an open cover of $\omega(\phi_\tau)$ and because of the compactness of $\omega(\phi_\tau)$ in $H^1$, we can find a finite sub-cover $\{\textbf{B}_{\beta_i}(\phi_\infty^i):i=1,2,...,m\}$ of $\omega(\phi_\tau)$ in $H^1$, where the constants $\beta_i, \theta_i$ corresponding to $\phi_\infty^i$ in Lemma \ref{LS} are indexed by $i$. From the definition of $\omega(\phi_\tau)$, there exists a sufficient large $t_0>\max\{\tau, 0\}$ such that
\beq\nonumber \phi(t)\in\mathcal{U}:=\bigcup_{i=1}^m\textbf{B}_{\beta_i}(\psi_{i}), \quad \text{for}\;\; t\geq t_0. \eeq
Taking $\theta=\min_{i=1}^m\{\theta_i\}\in(0,\frac{1}{2})$, using Lemma \ref{LS} and the convergence of energy \eqref{limE},  we deduce that for all $t\geq t_0$,
\beq\label{ls1}
\|\mathrm{P}_0(-\Delta\phi+f'(\phi))\|\geq |E(\phi(t))-E_{\infty}|^{1-\theta}.
\eeq
It follows from \eqref{BEL} and \eqref{dt} that
\beq
\frac{d}{dt}E(\phi(t))+\frac{1}{4K_2}\|\phi_t\|_{(H^1(\Omega))'}^2+\frac14\|\nabla\mu\|^2+\frac34\|\ub\|^2\leq \left(K_1+\frac14\right)\|S\|^2,\quad\text{for a.e.}\ t\geq \tau.
\label{BEL1}
\eeq
Introduce the auxiliary functions
\beq\nonumber
\mathcal{Y}(t)^2=\frac{1}{4K_2}\|\phi_t\|_{(H^1(\Omega))'}^2+\frac14\|\nabla\mu\|^2+\frac34\|\ub\|^2,\quad  z(t)=\left(K_1+\frac14\right)\int_t^{\infty}\|S\|^2ds.
\eeq
The assumption \eqref{AS} implies that
\beq\nonumber
z(t)\leq C(1+t)^{-(1+\rho)},\quad \forall t\geq t_0.
\eeq
Then the energy inequality \eqref{BEL1} yields that  for $t\geq t_0$,
\bea
 E(\phi(t))-E_{\infty}&\geq& \int_{t}^{\infty}\mathcal{Y}(s)^2ds-z(t)\nonumber\\
&\geq&\int_t^{\infty}\mathcal{Y}(s)^2 ds-C(1+t)^{-(1+\rho)}.\label{eee0}
\eea
Set the exponent
\beq\nonumber
\zeta=\min\left\{\theta,\frac{\rho}{2(1+\rho)}\right\}\in(0,\frac{1}{2}).
\eeq
We infer from \eqref{ls1} and the uniform $H^2$-bound for $\phi$ that
\bea
|E(\phi(t))-E_{\infty}|&\leq&\|\mathrm{P}_0(-\Delta\phi+f'(\phi))\|^{\frac{1}{1-\theta}}\non\\
&\leq&C \|\mathrm{P}_0(-\Delta\phi+f'(\phi))\|^{\frac{1}{1-\zeta}}\non\\
&\leq&C\|\nabla\mu\|^{\frac{1}{1-\zeta}}\leq C\mathcal{Y}(t)^{\frac{1}{1-\zeta}},\quad \forall t\geq t_0.
\label{ee}
\eea
 On the other hand, it is easy to verify that
\beq
\int_t^{\infty}(1+s)^{-2(1+\rho)(1-\zeta)}ds\leq\int_t^{\infty}(1+s)^{-(2+\rho)}ds\leq (1+t)^{-(1+\rho)},\quad \forall t\geq t_0.\label{eee1}
\eeq
Now we denote
\beq\nonumber
Z(t)=\mathcal{Y}(t)+(1+t)^{-(1+\rho)(1-\zeta)}.
\eeq
It follows from \eqref{eee0}--\eqref{eee1} that
\bea\int_t^{\infty}Z(s)^2 ds&\leq& C\mathcal{Y}(t)^{\frac{1}{1-\zeta}}+C(1+t)^{-(1+\rho)}\non\\
&\leq & CZ(t)^{\frac{1}{1-\zeta}},\quad\forall t\geq t_0.\label{z}
\eea
Thanks to the technical Lemma \ref{f}, we conclude from \eqref{z} that
\beq\label{z1}\int_{t_0}^{+\infty}Z(t)dt<+\infty.\eeq
Since $\rho>0$, we also have
\beq\nonumber\int_{t_0}^{+\infty}(1+t)^{-(1+\rho)(1-\zeta)}dt\leq\int_{t_0}^{\infty}(1+t)^{-\frac{2+\rho}{2}}dt
=\frac{2}{\rho}(1+t_0)^{-\frac{\rho}{2}}<+\infty,\quad\text{for  }t_0>0,
\eeq
which together with \eqref{z1} yields
\beq\nonumber\int_{t_0}^{+\infty}\|\phi_t\|_{(H^1(\Omega))'} dt<+\infty.\eeq
As a consequence, $\phi(t)$ converges strongly in $(H^1(\Omega))'$ as $t\rightarrow+\infty$. Together with the compactness of the trajectory in $H^s(\Omega)$, $s\in(0,2)$, we finally obtain that there exists  $\phi_{\infty}\in \mathcal{S}$ such that
\beq\nonumber
\lim_{t\rightarrow+\infty}\|\phi(t)-\phi_{\infty}\|_{H^s}=0
\quad \text{and} \quad \phi(t)\rightharpoonup \phi_\infty\ \text{weakly in}\ H^2(\Omega).
\eeq

Next, we proceed to prove the estimate on convergence rate. Let \beq\nonumber\mathcal{K}(t)=E(t)-E_{\infty}+z(t).\eeq
It follows from \eqref{BEL1} that
\beq
\label{ener1}
\frac{d}{dt}\mathcal{K}(t)+\mathcal{Y}(t)^2\leq 0,\quad\text{for   }t\geq t_0.
\eeq
Thus, $\mathcal{K}(t)$ is decreasing on $[t_0,+\infty)$ and due to \eqref{limE} and \eqref{AS}, $\mathcal{K}(t)\rightarrow0$ as $t\rightarrow+\infty$. Besides, we deduce from
%the fact  \beq\nonumber \frac{2(1-\zeta)}{1-\theta}>2.\eeq
%and
\eqref{AS}, \eqref{ee} that
\bea
\mathcal{K}(t)^{2(1-\theta)}&\leq& C\mathcal{Y}(t)^{2}+C(1+t)^{-2(1-\theta)(1+\rho)}\non\\
&\leq&-C\frac{d}{dt}\mathcal{K}(t)+C(1+t)^{-2(1-\theta)(1+\rho)}.\nonumber
\eea
Then by \cite[Lemma 2.6]{Ha}, we obtain that
\beq\nonumber \mathcal{K}(t)\leq C(1+t)^{-\kappa},\quad \forall t\geq t_0,\eeq
with the exponent given by
\beq\nonumber \kappa=\min\left\{\frac{1}{1-2\theta},1+\rho\right\}.\eeq
We infer from \eqref{ener1} that for any $t\geq t_0,$
\beq\nonumber\int_t^{2t}\mathcal{Y}(s)ds\leq t^{\frac{1}{2}}\left(\int_t^{2t}\mathcal{Y}^2(s)ds\right)^{\frac{1}{2}}\leq Ct^{\frac{1}{2}}\mathcal{K}^{\frac{1}{2}}(t)\leq C(1+t)^{\frac{1-\kappa}{2}}.\eeq
Thus, we have
\beq\nonumber \int_t^{+\infty}\mathcal{Y}(s)ds\leq \sum\limits_{j=0}^{+\infty}\int_{2^jt}^{2^{j+1}t}\mathcal{Y}(s)ds\leq C\sum\limits_{j=0}^{+\infty}(2^jt)^{-\lambda}\leq C(1+t)^{-\lambda},\quad\forall t\geq t_0,\eeq
where
\beq
\lambda=\frac{\kappa-1}{2}=\min\left\{\frac{\theta}{1-2\theta},\frac{\rho}{2}\right\}>0.\label{lamb}
\eeq
Therefore,
\beq\nonumber\int_t^{+\infty}\|\phi_t\|_{(H^1(\Omega))'} ds\leq C\int_t^{+\infty}\mathcal{Y}(s)ds\leq C(1+t)^{-\lambda},\quad\forall\, t\geq t_0, \eeq
which yields the convergence rate of $\phi$ in $(H^1(\Omega))'$:
\beq
\nonumber\|\phi(t)-\phi_{\infty}\|_{(H^1(\Omega))'}\leq C(1+t)^{-\lambda},\quad\forall\, t\geq t_0.
\eeq
Using the interpolation inequality and the uniform $H^2$-estimates for $\phi$, we have for any $s\in [-1,2]$,
\bea
\|\phi(t)-\phi_{\infty}\|_{H^{s}}&\leq& C\|\phi(t)-\phi_{\infty}\|_{(H^1(\Omega))'}^{\frac{2-s}{3}}\|\phi(t)-\phi_{\infty}\|_{H^{2}}^{\frac{s+1}{3}}\non\\
&\leq& C(1+t)^{-\frac{2-s}{3}\lambda},\quad\forall\, t\geq t_0.\label{rate1}
\eea
The proof of Theorem \ref{T3} is complete.

\begin{remark}
If the external source term $S$ is more regular, further decay property can be obtained. For instance, if in addition $S\in L^2(\tau,+\infty;\dot H^1(\Omega))\cap H^1(\tau,+\infty; \dot H^{-2}(\Omega))$, then using the energy method (see e.g., \cite{WW2012,ZWH,GWZ06}), we can prove
\beq\nonumber \lim_{t\to+\infty}(\|\phi(t)-\phi_\infty\|_{H^3}+\|\ub(t)\|+\|p(t)\|_{H^1})=0.\eeq
Moreover, the convergence rate \eqref{rate1} can be improved such that
\beq
\|\phi(t)-\phi_{\infty}\|_{H^2}\leq C(1+t)^{-\lambda},\quad\forall\, t\geq t_0,\non
\eeq
where the exponent $\lambda$ is given in \eqref{lamb}.
\end{remark}

\section{Appendix}

We first recall the following Gronwall-type inequality (see \cite[Lemma 2.5]{GGP}):
\begin{lemma}
\label{Gron0}
Let $y(t)$, $f(t)$ and $g(t)$ be nonnegative locally integrable functions on $[\tau,+\infty)$ which satisfy, for some $\gamma>0$
\beq
\frac{d}{dt}y(t)+\gamma y(t)\leq f(t)y^\frac12(t)+g(t) \qquad\text{for a.e. }\;\;t\in[\tau,+\infty).
\eeq
Then
\beq y(t)\leq2y(\tau)e^{-\gamma(t-\tau)}+\left(\int_{\tau}^tf(s)e^{-\frac{\gamma}{2}(t-s)}ds\right)^2+2\int_{\tau}^t g(s)e^{-\gamma(t-s)}ds\eeq
for any $t\in[\tau,+\infty)$. Moreover, the inequality
\beq
\int_{\tau}^t m(s)e^{-\gamma(t-s)}ds\leq\frac{e^{\gamma}}{1-e^{-\gamma}}\sup\limits_{r\geq\tau}\int_r^{r+1}m(s)ds
\eeq holds for any nonnegative locally
 integrable function $m$ on $[\tau,+\infty)$ and any $\gamma>0.$
\end{lemma}
The above lemma easily yields the following result

\begin{corollary}\label{Gron0a} Let $y(t)$, $f(t)$ and $g(t)$ be the nonnegative locally integrable functions on $[\tau,+\infty)$ that satisfy the assumptions in Lemma \ref{Gron0}. Assume, in addition that
\beq \sup\limits_{t\geq\tau}\int_t^{t+1}f(s)ds\leq A_1\qquad\text{and}\quad\sup\limits_{t\geq\tau}\int_t^{t+1}g(s)ds\leq A_2\eeq for some positive constants $A_1,A_2$.
Then
\beq\label{gronwall0} y(t)\leq2y(\tau)e^{-\gamma(t-\tau)}+Q(\gamma,A_1,A_2)\eeq where
\beq Q(\gamma,A_1,A_2)=\left(\frac{e^{\frac{\gamma}{2}}}{1-e^{-\frac{\gamma}{2}}}A_1\right)^2+\frac{2e^{\gamma}}{1-e^{-\gamma}}A_2.\eeq
\end{corollary}

The result in Corollary \ref{Gron0a} can be generalized. Namely, we have

\begin{lemma}\label{Gron1}
Let $y(t)$, $f(t)$ and $g(t)$ be nonnegative locally integrable functions on $[\tau,+\infty)$ which satisfy,
 for some $\gamma>0$ and some $\omega\in\{a_n\}_{n=0}^{\infty}$ with $a_n:=\frac{n+1}{n+2}$, $(n=0,1,2,...)$
\beq\label{gr1}\frac{d}{dt}y(t)+\gamma y(t)\leq f(t)y^{\omega}(t)+g(t) \qquad\text{for a.e. }\;\;t\in[\tau,+\infty)\eeq and such that
\beq\nonumber \sup\limits_{t\geq\tau}\int_t^{t+1}f(s)ds\leq A_1\qquad\text{and}\quad\sup\limits_{t\geq\tau}\int_t^{t+1}g(s)ds\leq A_2\eeq for
 some positive constants $A_1,A_2$. Then
\beq\label{gr0} y(t)\leq4\left(4^{\alpha_n}2^{\beta_n}y(\tau)e^{-\theta_n{\gamma}(t-\tau)}+Q^{\beta_n}(\frac{\gamma}{2},A_1,A_2)\right)\eeq
for any $t\in[\tau,+\infty)$, where
\beq\nonumber \alpha_n=
\begin{cases}
0,\qquad\qquad\qquad \quad  \!\text{if}\ n=0,\\
(n+2)\displaystyle{\sum\limits_{j=2}^{n+1}\frac{1}{j}},\qquad \text{if}\ n\geq 1,
\end{cases}
\quad \beta_n=\frac{n+2}{2},\quad\theta_n=\frac{n+2}{2^{n+1}},
\eeq and $Q$ is the same as in Lemma \ref{Gron0}.
\end{lemma}
\begin{proof}
Without loss of generality, we suppose that
$y(t)\geq 1$. Otherwise, we can simply set $\tilde{y}(t)=y(t)+1$. Using the fact $y^{\omega}<\tilde{y}^{\omega}$, we obtain a differential inequality for $\tilde{y}$ that has the same form as for $y$.

Then we prove the result by induction. The case $\omega=a_0=\frac{1}{2}$ corresponds to \eqref{gronwall0} in Corollary \ref{Gron0a}, with $\alpha_0=0$, $\beta_0=1$ and $\theta_0=\frac12$. Supposing that \eqref{gr0} holds for $\omega=a_n$ ($n\geq 0$), we consider the case $\omega=a_{n+1}$. Denote
$\varphi(t)=y^{\omega}(t)$. Then $y(t)=\varphi^{\frac{1}{\omega}}(t)$ and it holds that
\beq\nonumber\frac{d}{dt} \varphi(t)+\omega\gamma\varphi(t)\leq\omega f(t)\varphi^{2-\frac{1}{\omega}}(t)+\omega h(t),\eeq
where $$h(t)=\varphi^{1-\frac{1}{\omega}}(t)g(t).$$ Noticing that $\omega\in[\frac{1}{2},1)$, $\varphi(t)\geq 1$ and $2-\frac{1}{a_{n+1}}=a_n$, we have
$$h(t)\leq g(t)$$
and
\beq\nonumber \frac{d}{dt} \varphi(t)+\frac{\gamma}{2}\varphi(t)\leq f(t)\varphi^{a_n}(t)+\omega g(t).\eeq
Then it follows from the case  $\omega=a_n$ that
\beq\nonumber\varphi(t)\leq 4\left(4^{\alpha_n}2^{\beta_n}\varphi(\tau)e^{-\frac{\theta_n{\gamma}(t-\tau)}{2}}+Q^{\beta_n}(\frac{\gamma}{2},A_1,A_2)\right)\eeq
i.e.,
\beq\nonumber y^{\omega}(t)\leq4\left(4^{\alpha_n}2^{\beta_n}y^{\omega}(\tau)e^{-\frac{\theta_n{\gamma}(t-\tau)}{2}}+Q^{\beta_n}(\frac{\gamma}{2},A_1,A_2)\right).\eeq
Applying the elementary inequality\beq\nonumber(x+y)^{\theta}\leq 4(x^{\theta}+y^{\theta}), \qquad\text{for }x,y>0,\;1\leq\theta\leq 2\eeq
and noticing that $\frac{1}{\omega}\in(1,2]$, we get
\beq\nonumber
\begin{split}y(t)\leq&4\left(4^\frac{(1+\alpha_n)}{a_{n+1}}2^\frac{\beta_n}{a_{n+1}}y(\tau)e^{-\frac{(t-\tau)\gamma\theta_{n} }{2a_{n+1}}}+Q^\frac{\beta_n}{a_{n+1}}(\frac{\gamma}{2},A_1,A_2)\right)\end{split},\eeq
with
\beq\nonumber
\alpha_{n+1}=\frac{1+\alpha_n}{a_{n+1}},\quad \beta_{n+1}=\frac{\beta_n}{a_{n+1}},\quad \theta_{n+1}=\frac{\theta_{n} }{2a_{n+1}},
\eeq
such that \eqref{gr0} holds for $\omega=a_{n+1}$. This completes the proof.
\end{proof}
\begin{remark}
Since $a_n\nearrow1$ as $n\rightarrow+\infty,$ the above lemma enables
us to deal with the general case $\omega\in(\frac{1}{2},1)$ in \eqref{gr1}. On the other hand, when $\omega\in (0,\frac{1}{2})$, we can also employ Lemma \ref{Gron0}, thanks to Young's inequality such that $y^{\omega}\leq 2\omega y^{\frac{1}{2}}+(1-2\omega)$.
\end{remark}

The following lemma (cf. \cite{FS, HT01}) will be used to study the long-time behavior of global solutions to problem \eqref{1}--\eqref{ini}:
\begin{lemma}\label{f}
Let $\zeta\in(0,\frac{1}{2})$. Assume that $Z\geq0$ be a measurable function on $(\tau,+\infty)$, $Z\in L^2(\tau, +\infty)$ and there exist $C>0$ and
$t_0\geq\tau$ such that
\beq\nonumber \int_t^{\infty} Z^2(s)ds\leq CZ(t)^{\frac{1}{1-\zeta}},\quad\text{for a.e.   }t\geq t_0.\eeq
Then $Z\in L^1(t_0,+\infty).$
\end{lemma}

%%%%%%%%%%%%%%%%%%%%%%%%%%%%%%%%%%%%%%%%%%%%%%%%%%%%%%%%

\bigskip

\noindent\textbf{Acknowledgments:} J. Jiang is partially supported by National Natural Science Foundation of China (NNSFC) under the grants No. 11201468, H. Wu is partially supported by NNSFC under the grant No. 11371098 and Zhuo Xue program in Fudan University, and S. Zheng is partially supported by NNSFC under the grant No. 11131005.

%%%%%%%%%%%%%%%%%%%%%%%%%%%%%%%%%%%%%%%%%%%%%%%%%%%%%%%%

\end{document}